\documentclass[reqno]{amsproc}
\usepackage[utf8]{inputenc}
\usepackage{amsfonts}
\usepackage{graphicx}
\usepackage{verbatim}
\usepackage{amscd}
\usepackage{amsmath, physics}
\usepackage{amssymb}
\usepackage{latexsym}
\usepackage{hyperref}
\usepackage[all]{xy}
\usepackage{color}
\usepackage{mathrsfs}
\theoremstyle{definition}
\newtheorem{remark}{\bf Remark}

\newtheorem{example}{\bf Example}

\theoremstyle{plain}
\newtheorem{corollary}{\bf Corollary}
\newtheorem{lemma}{\bf Lemma}

\newtheorem{proposition}{\bf Proposition}

\newtheorem{theorem}{\bf Theorem}
{}
\numberwithin{equation}{section}

\DeclareMathOperator{\Id}{Id}
\DeclareMathOperator{\Ric}{Ric}
\DeclareMathOperator{\Area}{Area}
\usepackage{empheq,etoolbox}
\patchcmd{\subequations}% <cmd>
  {\theparentequation\alph{equation}}% <search>
  {\theparentequation.\alph{equation}}% <replace>
  {}{}% <success><failure>
\hypersetup{
  colorlinks=true,		% false: boxed links; true: colored links
  linkcolor=blue,		% color of internal links
  citecolor=blue,		% color of links to bibliography
  filecolor=blue,		% color of file links
  urlcolor=blue,		% color of internet links
 bookmarksdepth=4
}
\begin{document}

\title[Mean curvature flow in an extended Ricci flow background]{Mean curvature flow in an extended Ricci flow background}
\author[José Gomes]{José N. V. Gomes$^1$}
\author[Matheus Hudson]{Matheus Hudson$^2$}
\address{$^1$Programa de Pós-Graduação em Matemática (PPGM), Centro de Ciências Exatas e Tecnologia (CCET), Universidade Federal de São Carlos, São Carlos, São Paulo, Brazil.}
\address{$^2$Programa de Doutorado em Matemática (PDM) em Associação Ampla UFPA-UFAM, Instituto de Ciências Exatas (ICE), Universidade Federal do Amazonas, Manaus, Amazonas, Brazil.}
\email{$^1$jnvgomes@ufscar.br; jnvgomes@pq.cnpq.br}
\email{$^2$matheushudson10@gmail.com}
\urladdr{$^1$https://www2.ufscar.br}
\urladdr{$^2$https://ufam.edu.br}

\begin{abstract}
 In this paper, we consider functionals related to mean curvature flow in an ambient space which evolves by an extended Ricci flow from the perspective introduced by Lott when studying a mean curvature flow in a Ricci flow background. One of them is a  weighted extended version of the Gibbons-Hawking-York action  on Riemannian metrics in compact manifolds with boundary. We compute its variational properties from which naturally arise boundary conditions to the analysis of its time-derivative under Perelman's modified extended Ricci flow. For instance, the boundary integrand term provides an extension of Hamilton's differential Harnack expression for mean curvature flows in Euclidean space. We also derive the evolution equations for both the second fundamental form and the mean curvature under mean curvature flow in an extended Ricci flow background. In the special case of gradient solitons to the extended Ricci flow, we discuss mean curvature solitons and establish a Huisken's monotonicity-type formula. We show how to construct a family of mean curvature solitons and establish a characterization of such a family. Also, we show how for constructing examples of mean curvature solitons in an extended Ricci flow background.
\end{abstract}
\maketitle

\section{Introduction}
One of the greatest mathematical achievements of this century was the proof of Thurston’s Geometrization Conjecture, by Perelman, which, as a consequence, settled affirmatively the celebrated Poincaré's Conjecture. The main tool used by Perelman in his proofs was the \emph{Ricci flow}, introduced  by Hamilton~\cite{Hamilton92}, which is defined as follows. Let $\mathcal G:=\{g(t)\}_{t\in[a,b]}$ be a smoothly varying family of Riemannian metrics on an  $n(\geqslant 3)$-dimensional compact smooth manifold $M.$ One says that $\mathcal G$ satisfies the Ricci flow equation if
\begin{equation} \label{eq-ricciflow}
\frac{\partial}{\partial t} g(t)=-2\Ric_{g(t)} \,\,\, \forall t\in[a,b],
\end{equation}
where $\Ric_{g(t)}$ denotes the Ricci curvature of $(M,g(t)).$

Hamilton established existence and uniqueness of solutions to \eqref{eq-ricciflow} in a maximal interval $[0,T),$ $T\leqslant +\infty$, for any given initial  metric $g=g(0).$ This maximal solution is then called the \emph{Ricci flow} with initial condition $g,$ and $T$ (whenever finite)  is called the \emph{blow-up time} of the flow.

An important geometric flow  also considered by Hamilton~\cite{Hamilton} was the celebrated \emph{mean curvature flow} (MCF, for short), which falls in the class of \emph{extrinsic geometric flows} (see \eqref{mean:sys}). Since then, mean curvature flow has been a constant object of investigation, and  has experienced a great development in the last decades. It should also be mentioned that MCF has   applications in many fields, including  geometric analysis, geometric measure theory, and partial differential equations, to name a few.

A significant contribution given by  G. Perelman in the study of the Ricci flow on smooth manifolds was the discovery of its gradient-like structure, namely, he showed how the Ricci flow can be regarded as a gradient flow  from its $\mathcal F$-functional on compact manifolds with weighted preserving-measure (see \cite[Sects.~1~and~3]{Perelman} and~\cite[Sects.~10~and~12]{Kleiner_Lott}).  Moreover, he defined an associated entropy by means of its $\mathcal W$-functional.

In a similar way, List showed how the extended Ricci flow (see~\eqref{ListEq}) can also be regarded as a gradient flow (cf.~\cite[Lem.~3.4~and~Thm.~6.1]{Bernhard_List}).  Moreover, he proved the existence of a Perelman $\mathcal F$-type functional such that the stationary points are solutions to the static Einstein vacuum equations and studied an extended parabolic system  which is equivalent to the gradient flow of his functional. However, while in such flows on manifolds with boundary, one loses the sense of gradient-like structures,  interesting properties appear when the boundary evolves by some geometric flow.

Ecker defined a version of Perelman's $\mathcal W$-functional for Ricci flow on bounded domains with smooth boundary. In its time-derivative appears Hamilton's differential Harnack expression on the boundary integrand (see~\cite{Hamilton}). In this work, he conjectured that his functional is nondecreasing in time under the mean curvature flow of any compact hypersurface in $\mathbb{R}^n,$ see~~\cite[Sects.~2 and 3]{Klaus-Ecker} for the definition of Ecker's functional and its time-derivative.

Inspired by Ecker's work, Lott~\cite{John_Lott} approached mean curvature flow in arbitrary  Ricci flow background by introducing an analogue of Perelman's $\mathcal F$-functional for a manifold $M$ with boundary $\partial M$. More precisely, he added a boundary term to the interior integral of $\mathcal F$, obtaining then a weighted version $I_\infty$ of the Gibbons-Hawking-York action~\cite{Gibbons-Hawking, York}, see also Araújo~\cite{Ara03}. In a similar way, one can think of an analogous conjecture for weighted Gibbons-Hawking-York action $I_\infty$ under the mean curvature flow in a Ricci flow  background, which is still an open problem. In both cases, an answer for these open problems required a study on the boundary integrand of the time-derivative of these functionals. In this setting, the main results obtained by Lott include the determination of the evolution equations of the action $I_\infty,$ of the second fundamental form of $\partial M,$ and of the mean curvature of $\partial M$ under Perelman's modified Ricci flow.

In the present paper, we intend to consider Lott's program in the context of mean curvature flow in an extended Ricci flow background. To be more precise, let $M$ be an $n (\geqslant 3)$-dimensional smooth manifold and  let $(g(t), w(t))$  be a solution  to the \emph{extended Ricci flow}
\begin{align}\label{ListEq}
\left\{
\begin{array}{lcl}
\frac{\partial}{\partial t}g(t) = - 2\Ric_{g(t)} + 2\alpha_n dw(t)\otimes dw(t),\\[1ex]
\frac{\partial}{\partial  t}w(t) =  \Delta_{g(t)} w(t),
\end{array}
\right.
\end{align}
in $M \times [0,T),$ for some initial value $(g,w).$ Here and throughout the paper, $\alpha_n=(n - 1)/(n - 2)$, ${\rm Ric}_{g(t)}$ stands for the Ricci tensor of the Riemannian metric $g(t)$, the Laplacian operator $\Delta_{g(t)}$ is the trace of the Hessian operator $\nabla_{g(t)}^2$ computed on $g(t)$, and $dw(t)\otimes dw(t)$ denotes the tensor product of the $1$-form $dw(t)$ by itself, which is metrically dual to gradient vector field $\nabla w(t)$ computed on $g(t)$ of a scalar smooth function $w(t)$ on $M.$ For on account on extended Ricci flows, including a proof of short-time existence of solutions to \eqref{ListEq}, we refer to List~\cite[Thm.~4.1]{Bernhard_List}.

A gradient soliton to the extended Ricci flow is, by definition, a self-similar solution $(\overline{g}(t), \overline{w}(t) )$ of~\eqref{ListEq} given by
\begin{align*}
\left\{
\begin{array}{lcl}
\overline{g}(t)  = \sigma(t)\psi_t^* g,\\ [1ex]
\overline{w}(t) = \psi_t^* w,
\end{array}
\right.
\end{align*}
for some initial value $(g,w)$, where $\psi_t$ is a  smooth one-parameter family of diffeomorphisms of $M$ generated from the flow of $\nabla_g f/\sigma(t)$ computed on $g$, $f\in C^\infty(M)$, and  $\sigma(t)$ is a smooth positive function on $t.$ By setting $\overline{f}(t)=\psi_t^*f$,  sistem~\eqref{ListEq} becomes
\begin{align}\label{mod_grad_Ricci_soliton}
\left\{
\begin{array}{lcl}
\Ric_{\overline{g}} + \nabla_{\overline{g}}^2\,\overline{f} - \alpha_n  d\overline{w} \otimes d\overline{w} = \dfrac{c}{2t}\overline{g},\\
\Delta_{\overline{g}} \overline{w} = \langle\nabla_{\overline{g}} \overline{f}, \nabla_{\overline{g}} \overline{w}\rangle_{\overline{g}},
\end{array}
\right.
\end{align}
where $c= 0$ in the steady case (for $t \in  \mathbb{R}$ and $\psi_0=\Id$), $c = -1$ in the shrinking case (for $t \in (-\infty, 0)$ and $\psi_{-1}=\Id$) and $c = -1$ in the expanding case (for $t \in (0,\infty)$ and $\psi_1=\Id$). Moreover,
\begin{align}\label{self-solution}
\dfrac{\partial }{\partial t} \overline{f} = \|\nabla_{\overline{g}} \overline{f}\|_{\overline{g}}^2\,.
\end{align}
The function $\overline{f}$ is  then  called the \emph{potential function}. For more details, see Subsect.~\ref{sect:grad:sol}.

We shall consider mean curvature flows in the following context: let $(g(t),w(t))$ be an extended Ricci flow  in $M \times [0, T)$. Given an $(n-1)$-dimensional smooth compact manifold $\Sigma$ without boundary, let $\{x(\cdot,t); t \in [0, T)\}$ be a smooth one-parameter family of immersions of $\Sigma$ into $M$.  For each $t\in[0,T),$ set  $x_t = x(\cdot, t)$ and $\Sigma_t$ for the hypersurface $x_t(\Sigma)$ of $(M,g(t)),$ i.e., $\Sigma_t:=(x_t(\Sigma), g(t))$, and suppose  that the family  $\mathscr F :=\{\Sigma_t\,;\, t\in[0,T)\}$ evolves under mean curvature flow
\begin{align}\label{mean:sys}
\left\{
\begin{array}{rcl}
\frac{\partial}{\partial t}x(p,t) &=&  H(p,t) e(p,t),\\ [1ex]
x(p, 0) &=& x_0(p),
\end{array}
\right.
\end{align}
where $H(p,t)$ and $e(p,t)$ are the mean curvature and the unit normal of $\Sigma_t$ at the point $p\in\Sigma$, respectively. In this setting, we say that $\mathscr{F}$ is the \emph{mean curvature flow in the $(g(t),w(t))$-extended Ricci flow background}. In the particular case $(g(t),w(t))=(\overline{g}(t),\overline{w}(t))$ is a gradient soliton to the extended Ricci flow on $M$ with potential function $\overline{f}$, a hypersurface $\Sigma_t\in \mathscr F$ is a \emph{mean curvature soliton}, if
\begin{align*}
H(p,t)+e(p,t)\overline{f}=0 \,\,\, \forall p\in\Sigma_t.
\end{align*}
Here, $e(\,\cdot\,,t)$ must be the inward unit normal vector field on $\Sigma_t.$

Now suppose that $M$ is an $n(\geqslant 3)$-dimensional compact smooth manifold with boundary $\partial M$. Let ${\rm met}(M)$ be the set of all metrics $g$ on $M.$ We define the  \emph{weighted extended Gibbons-Hawking-York (GHY, for short) action} $I^{\alpha_n}_\infty$  on the product $\mathscr P(M):={\rm met}(M)\times C^\infty (M)\times C^\infty(M)$ as
\begin{align}\label{MGHYF}
I^{\alpha_n}_\infty(g, f, w):=\int_{M} \Big(R_\infty -\alpha_n |\nabla w|^2 \Big)e^{-f} dV + 2\int_{\partial M}H_\infty e^{-f}dA,
\end{align}
where $R_\infty:= R + 2 \Delta f - |\nabla f|^2$ is the  \emph{weighted scalar curvature} of $g$, the function $H_\infty:=H + e_0 f$   is  the \emph{weighted mean curvature} with respect to the inward unit normal vector field $e_0$ on $\partial M$, the forms $dV$ and $dA$ are the $n$-dimensional Riemannian measure of $(M,g),$ and the $(n-1)$-dimensional Riemannian measure of $(\partial M,g),$ respectively.

The action $I^{\alpha_n}_\infty$ is the proper extension to our context of the action $I_\infty$ introduced by Lott in~\cite{John_Lott}. It should also be mentioned that the function $R_\infty$ arises quite naturally, as observed by Perelman~\cite[Sect.~1.3]{Perelman}, and $H_\infty$ is in fact the appropriate geometric object when we are using a weighted measure (see, e.g.,~\cite[Sect.~9.4.E]{Gromov} or \cite[Subsec.~3.3]{John_Lott}).

Our first main result extends \cite[Theorem 1]{John_Lott} to the context of mean curvature flow in an extended Ricci flow background. It reads as follows (see Sections~\ref{Sec:preliminaries} and~\ref{Sec:VWEGHYA} for definitions and notations).  

\begin{theorem}\label{principal_theorem}
Let $M$ be an $n(\geqslant 3)$-dimensional compact smooth manifold with boundary $\partial M$, and let $I^{\alpha_n}_\infty$ be the weighted extended GHY-action on  $\mathscr P(M)$ defined as in \eqref{MGHYF}. Suppose that the family $\{{\partial M}_t\,;\, t\in[0,T)\}$ is a MCF in the $(g(t),w(t))$-extended Ricci flow background which satisfies $e_0 w = 0$ on $\partial M,$ where $e_0$ is the inward unit normal vector field on $\partial M.$ Under these conditions, if $u:=e^{-f}$ is a solution to the conjugate heat equation
\begin{align}\label{back:heat:eq}
\frac{\partial}{\partial t}u=-\Delta u+Ru-\alpha_n |\nabla w|^2u
\end{align}
in $M\times[0,T)$, with $e_0 u=Hu$ on $\partial M,$ then
\begin{align*}
\dfrac{d}{d t}I_\infty^{\alpha_n} &= 2 \int_{M} \Big(|{\rm Ric}\, + \nabla^2 f -\alpha_n dw \otimes dw|^2 + \alpha_n\big(\Delta w - \langle \nabla w, \nabla f \rangle\big)^2\Big) e^{-f}dV \\
&\quad + 2 \int_{\partial M} \Big(\frac{\partial H}{\partial t}  - 2\langle \widehat{\nabla} f, \widehat{\nabla} H\rangle + \mathcal{A}
(\widehat{\nabla} f, \widehat{\nabla} f)  + 2 R^{0 i}\widehat{\nabla}_i f  - \dfrac{1}{2} \nabla_0 R -HR_{00}   \\
&\quad +  \alpha_n    \mathcal{A}( \widehat{\nabla} w, \widehat{\nabla} w) \Big)  e^{-f} dA,
\end{align*}
where $\mathcal{A}$ is the second fundamental form of $\partial M,$ and $\widehat{\nabla}f$ denotes the gradient of $f$ on $(\partial M,g(t)).$ 
\end{theorem}

For the proof of Theorem~\ref{principal_theorem}, we first study the \emph{Perelman’s modified extended Ricci flow} (see Section~\ref{Sec:VWEGHYA}), and then ``translate'' the results for the context of  extended Ricci flow. Also, as an application of Theorem~\ref{principal_theorem}, we obtain an extension  of Hamilton’s differential Harnack expression for mean curvature flow in Euclidean space to the more general context of mean curvature flow in a gradient steady soliton  to the extended Ricci flow (cf. Corollary~\ref{type-Harnack-background} in Section~\ref{se: Ricci}). 

Our second main result is a Huisken’s monotonicity-type formula~\cite{Huisken1990} for the mean curvature flow in an extended Ricci flow background, as stated below.
\begin{theorem}\label{Huisken_monotonicity}
Let $M$ be an $n(\geqslant3)$-dimensional smooth manifold. Let $(\overline g(t),\overline w(t))$ be a gradient soliton to the extended Ricci flow on $M$ with potential function $\overline{f}$. Assume that $\mathscr F:=\{\Sigma_t\}$ is a MCF in the $(\overline g,\overline w )$-extended Ricci flow background, denote by $dA_{\overline{g}}$ the $(n-1)$-dimensional Riemannian measure on $\Sigma_t$ induced by $\overline g(t)$, and set $\Area_{\overline{f}}(\Sigma_t):=\int_{\Sigma_t} e^{-\overline{f}}dA_{\overline{g}}$. Under these conditions, the function $\Phi(t)$ given by:
\begin{itemize}
    \item[(i)]  $\mathbb{R}\ni t\mapsto\Area_{\overline{f}}(\Sigma_t)$ in the steady case;
    \item[(ii)] $(-\infty,0)\ni t\mapsto (-t)^{-( n - 1 ) / 2}\Area_{\overline{f}}(\Sigma_t)$ in the shrinking case;
    \item[(iii)] $(0,\infty)\ni t\mapsto t^{-( n - 1 ) / 2}\Area_{\overline{f}}(\Sigma_t)$ in the expanding case;
\end{itemize}
is nonincreasing. Moreover, $\Phi(t)$ is constant if and only if $\mathscr F$ is a family of mean curvature solitons.
\end{theorem}

In the Euclidean space case, Huisken proved that shrinking self-similar solutions to the mean curvature flow are exactly singularity of type-I (i.e., the growth rate estimate for the norm of the second fundamental form is bounded) and asymptotically self-similar which appears as stationary points for the Gaussian area-type functional playing the role of the energy-type functional (see~\cite{Huisken1990} for details). Theorem~\ref{Huisken_monotonicity} is a useful tool for studying an analogous to Huiken's result to mean curvature solitons in the $(\overline g,\overline w )$-extended Ricci flow background.
 
We also show how to construct a family of mean curvature solitons and establish a characterization of such a family. In particular, we highlight the content of our Theorem~\ref{NazasCharacterization}, which guarantees that a $f$-minimal hypersurface $\Sigma$ of $(M,g)$ can be evolved as a family $\mathscr G$ of mean curvature solitons. Moreover, any family $\mathscr F$ of mean curvature solitons is given by $\mathscr G$ up to reparametrization.

We point out that, by considering particular cases of our results (for instance, assuming $g(t)$  or $w(t)$ constant), we recover some previous results on mean curvature flows (see Remarks \ref{rem: recover1}, \ref{rem: recover2}, \ref{rem: recover3}, \ref{rem: recover4} and \ref{rem: recover5}).

\section{Preliminaries} \label{Sec:preliminaries}

Throughout the paper, all   manifolds  are assumed to be orientable and connected. Also, in dealing with flows, we shall usually simplify the notation by suppressing the parameter $t.$

We shall adopt the following notation. Given an $n$-dimensional manifold $M$ with nonempty boundary $\partial M,$ we shall denote the local coordinates at $p\in M$ by $\{x^\alpha\}_{\alpha = 0}^{n-1}$ and  the local coordinate basis by $\{\partial_{\alpha}\}_{\alpha = 0}^{n - 1}$. Near $\partial M$, we take $x^0$ to be a local defining function for $\partial M$. We denote the local coordinates for $\partial M$ by $\{x_i\}^{n-1}_{i=1}$. We choose these coordinates near a point at $\partial M$ so that $\partial_0 \big|_{\partial M}$ coincides with the inward-pointing unit normal field $e_0$ on $\partial M$. The Greek letters $\alpha, \beta, \ldots$ stand for the indices associated with the coordinates on $M$, whereas $i, j, \ldots$ stand for the indices of the coordinates on $\partial M$.

For a Riemannian metric $g$ on $M,$ we denote by $\nabla$ the Levi-Civita connection on $TM$ and by $\widehat{\nabla}$ the Levi-Civita connection on $T\partial M$. The curvature tensor of $g$ in local coordinates is given by
$$R^\zeta\!_{\alpha\beta\gamma}\partial_\zeta=R(\partial_\alpha,\partial_\beta)\partial_\gamma= \nabla_\beta \nabla_\alpha \partial_\gamma -
\nabla_\alpha \nabla_\beta \partial_\gamma,$$
so that $R^\zeta\!_{\alpha\beta\gamma} = g^{\xi \zeta} R_{ \alpha \beta\gamma \xi},$ where $R_{\alpha \beta\gamma\xi} = g ( R(\partial_\alpha,\partial_\beta)\partial_\gamma, \partial_\xi).$

We also use the classical notations $v^{\alpha\beta} = g^{\alpha \gamma} g^{\beta\zeta} v_{\gamma\zeta}$ for any $2$-tensor field $v$, which yields the equalities 
$$R^{i0}= g^{ik}g^{0\alpha} R_{\alpha k}=g^{ik}R^0\!_k, \quad R_{\alpha\beta}:={\rm Ric}(\partial_\alpha, \partial_\beta).$$
In addition, we set $\nabla^\alpha w= g^{\alpha\beta}\nabla_\beta w = g^{\alpha\beta}g(\nabla w, \partial_\beta),$ which gives the following expression for the Hessian tensor:
$$\nabla^\alpha \nabla^\beta f = g^{\alpha \gamma}g^{\beta\zeta} \nabla_\gamma \nabla_\zeta f, \quad \nabla_\gamma\nabla_\zeta f:=\nabla^2 f (\partial_\gamma, \partial_\zeta).$$

In what concerns $\partial M,$ we write $\mathcal{A}_{ij}:=g(\nabla_{\partial_i} \partial_j,e_0)$ for its second fundamental form, and $H:=g^{ij}\mathcal{A}_{ij}$ for its mean curvature.  Hence,
$$\mathcal{A}^{ij} = g^{ik}g^{jl}\mathcal{A}_{kl} \quad\text{and}\quad \mathcal{A}^k\!_i=g^{kl}\mathcal{A}_{li}.$$

\section{Evolution of the weighted extended GHY-action under Perelman's modified extended Ricci flow}\label{Sec:VWEGHYA}

In this section, we obtain a variational formula for the weighted extended GHY-action $I^{\alpha_n}_\infty$ on  $\mathscr P(M):={\rm met}(M)\times C^\infty (M)\times C^\infty(M),$ where $M$ is an $n(\geqslant 3)$-dimensional smooth manifold with boundary $\partial M.$

We shall adopt the following notation. Given $(g,f,w)\in \mathscr P(M)$ and a variation $\delta g_{\alpha\beta}=v_{\alpha\beta}$ of $g,$ we shall denote by $\delta f=h$ and by $\delta w=\vartheta$ the variations of $f$ and $w$, respectively, and  write $v=g^{\alpha \beta}v_{\alpha \beta}.$ Note that the volume element $e^{-f}dV$ is measure-preserving if and only if $\frac{v}{2}-h=0$ on $M$, since $\delta(e^{-f} d V)=(\frac{v}{2}-h)e^{-f}dV$.

\begin{proposition}\label{var:MWGHY}
With the above notations, suppose that $\frac{v}{2}-h=0$ on $M.$ Then, the following equality holds:
\begin{align*}
\delta I^{\alpha_n}_\infty \!&=\! \int_{M}\!\! \Big( v^{\alpha\beta}\big(\alpha_n \nabla_\alpha w \nabla_\beta w
- R_{\alpha\beta} - \nabla_\alpha\nabla_\beta f \big) +2\alpha_n\vartheta\big(\Delta w - \langle \nabla w, \nabla f\rangle \big)\Big) e^{-f}dV\\
&\quad - \int_{\partial M} \big( v^{ij}\mathcal{A}_{ij} + v^{00}(H + e_0 f)\big) e^{-f} d A + 2\alpha_n \int_{\partial M}\vartheta e_0 w   e^{-f} dA.
\end{align*}

\begin{proof}
First, we observe that by~\eqref{MGHYF}
\begin{align*}
I^{\alpha_n}_\infty(g, f, w) = I_\infty(g, f) - \alpha_n  I_1(g, f, w),
\end{align*}
where $I_1(g, f, w) :=\int_{M} |\nabla w|^2 e^{-f}dV$. Notice that
\begin{align*}
\delta I_1= \int_{M} \Big(\delta\big(|\nabla w|^2\big) + |\nabla w|^2\Big(\frac{v}{2} - h\Big)\Big) e^{-f} dV.
\end{align*}
A straightforward computation gives us
\begin{align*}
\delta\big(|\nabla w|^2\big) = - g^{\alpha\gamma}v_{\gamma \zeta}g^{\beta\zeta}\nabla_\alpha w \nabla_\beta w + 2 g^{\alpha\beta} \nabla_\alpha\vartheta \nabla_\beta w.
\end{align*}
So,
\begin{align*}
\delta I_1= \int_{M}\Big(- v^{\alpha\beta}\nabla_\alpha w \nabla_\beta w e^{-f} + 2 g^{\alpha\beta} \nabla_\alpha\vartheta \nabla_\beta w e^{-f}\Big) dV.\\
\end{align*}
Using integration by parts, we have
\begin{align*}
\delta I_1 &= \int_{M}\Big(- v^{\alpha\beta}\nabla_\alpha w \nabla_\beta w e^{-f}+ 2 g^{\alpha \beta}\nabla_\alpha \Big(\vartheta\nabla_\beta w e^{-f} \Big) -2g^{\alpha \beta}\vartheta \nabla_\alpha \nabla_\beta w e^{-f}\\
&\quad -2g^{\alpha \beta}\vartheta \nabla_{\gamma}w \Gamma^{\gamma}_{\alpha\beta} e^{-f} + 2g^{\alpha \beta}\vartheta \nabla_\beta w \nabla_\alpha f  e^{-f} \Big) dV.
\end{align*}
Since
\begin{align*}
{\rm div} (\vartheta e^{-f}\nabla w) = g^{\alpha \beta}\nabla_\alpha \Big(\vartheta\nabla_\beta w e^{-f} \Big)-g^{\alpha \beta}\vartheta\nabla_{\gamma}w \Gamma^{\gamma}_{\alpha\beta} e^{-f},
\end{align*}
by Stokes' theorem, we get
\begin{align*}
\delta I_1&=\int_{M} \Big( - v^{\alpha\beta}\nabla_\alpha w \nabla_\beta w -2 \vartheta g^{\alpha \beta}\nabla_\alpha \nabla_\beta w e^{-f} + 2  \vartheta g^{\alpha\beta}\nabla_\beta w \nabla_\alpha f   \Big)e^{-f} dV\\
&\quad - 2\int_{\partial M} \vartheta  e_0 w \ e^{-f} d A.
\end{align*}
Proposition~2 in \cite{John_Lott} guarantees that
\begin{align*}
\delta I_{\infty}=&-\int_M v^{\alpha \beta}\left(R_{\alpha \beta}+\nabla_\alpha \nabla_\beta f\right) e^{-f} d V -\int_{\partial M}\left(v^{i j} \mathcal{A}_{i j}+v^{00}\left(H+e_0 f\right)\right) e^{-f} dA.
\end{align*}
The result of the proposition follows from these two previous equations.
\end{proof}
\end{proposition}
\begin{remark} \label{rem: recover1}
By considering $M$ compact without boundary in Proposition~\ref{var:MWGHY}, we recover the result by List~\cite[Eq.~(3.2)]{Bernhard_List}.
\end{remark}

We say that a family  $(g(t),w(t))\in{\rm met}(M)\times C^\infty (M),$ $t\in[0,T),$ is a \emph{Perelman's modified extended Ricci flow} if it satisfies the equations: 
 \begin{subequations}
\begin{empheq}[left=\empheqlbrace]{align}
\frac{\partial}{\partial t}g &= - 2(\Ric + \nabla^2 f - \alpha_n dw \otimes  dw),\label{grad:form:eq1}\\
\frac{\partial}{\partial t}w &=  \Delta w -\langle \nabla f, \nabla w\rangle.\label{grad:form:eq2}
\end{empheq}
\end{subequations}
and
\begin{align}\label{weighted_pr_meas}
\frac{\partial}{\partial t}f =-\Delta f-R+\alpha_n |\nabla w|^2
\end{align}
in $M\times[0,T)$, with $H+e_0 f=0$ and $e_0w = 0$ on $\partial M.$ In this case, the measure $e^{-f}dV$ remains fixed on $M.$

List~\cite[Sect.~2.1]{Bernhard_List_Thesis} showed that the Perelman's modified extended Ricci flow arises quite naturally in the study of structure-type gradient of its $\mathcal F$-functional on 
\[\mathcal{C}_1:=\Big\{(v, \vartheta) \in S^2(M) \times  C^\infty(M): v = \delta g \ , \ \vartheta = \delta w \ \mbox{and} \ \delta(e^{-f} d V)=0\Big\}.\]

Lott~\cite{John_Lott} showed that, on variations that fix both the measure $e^{-f}dV$ of $(M,g)$ and the induced metric $g_{\partial M}$ on $\partial M$, the critical points of $I_\infty$ are gradient steady Ricci solitons on $M$ that satisfy $H + e_0 f = 0$ on $\partial M$, thus, in this background, the boundary arise naturally as a $f$-minimal hypersurface, i.e., $H_\infty = 0$. Recall also that $f$-minimal hypersurfaces arise as critical points of the $f$-weighted Area functional. Some results concerning $f$-minimal hypersurfaces can be found, e.g., in \cite{alias2020mean,cheng2015stability,cheng2015simons,cheng2020minimal,wei2017lower,sheng2015}. 

From the next two corollaries, we know the critical points of the weighted extended GHY action on $\mathcal{C}_1$.
\begin{corollary}\label{Rem:WPM}
Let $M$ be an $n(\geqslant 3)$-dimensional compact smooth manifold with boundary $\partial M,$ and let  $I^{\alpha_n}_\infty$ be the weighted extended GHY-action
on  $\mathscr P(M)$ defined in~\eqref{MGHYF}. If the induced metric on $\partial M$ is fixed, then the critical points of $I_\infty^{\alpha_n}$ on $\mathcal{C}_1$ are gradient steady solitons on $M$ that satisfy $H +  e_0 f = 0$ and $e_0 w = 0$ on $\partial M$. 
\end{corollary}
\begin{proof}
By hypotheses we have $\frac{v^{\alpha}\!_{\alpha}}{2} - h = 0$ on $M$ and $v_{ij} = 0$ on $\partial M$ which allows us use Proposition~\ref{var:MWGHY} to obtain
 \begin{align}\label{var:almost}
 &\int_{M}\Big(\langle v,  \alpha_n d w \otimes d w - \Ric_g - \nabla^2_g f\rangle +2\alpha_n\vartheta\big(\Delta_g w -  \langle \nabla w, \nabla f\rangle \big) \Big) e^{-f} dV \nonumber \\
 &\quad+\int_{\partial M}\Big( 2\alpha_n \vartheta e_0 w - \langle v, (H + e_0 f) e^\flat_0 \otimes e^\flat_0\rangle  \Big)e^{-f} dA = 0,
\end{align}  
for all 
\begin{comment}
$v_{\al\bt}, \vartheta \in C^\infty(M).$ 
$v \in S^2(M), \vartheta \in C^\infty(M),$ 
\end{comment}
$(v, \vartheta) \in {\mathcal C}_1,$
where $  `` ^\flat "$ stands for musical isomorphism.  We first start assuming 
$(v, \vartheta) \in {\mathcal C_1}_c$.
 Then
\begin{align*}
 &\int_{M}\Big(\langle v,  \alpha_n d w \otimes d w - \Ric_g - \nabla^2_g f\rangle +2\alpha_n\vartheta\big(\Delta_g w -  \langle \nabla w, \nabla f\rangle \big) \Big) e^{-f} dV = 0.
\end{align*}  
Therefore $(g,w)$ must be a gradient steady soliton to the extended Ricci flow. Again using \eqref{var:almost} to obtain
\begin{align*}
\int_{\partial M} \langle v, (H + e_0 f) e^\flat_0 \otimes e^\flat_0\rangle e^{-f} dA = 0,
\end{align*} 
 so $H + e_0 f = 0$ on $\partial M$. To finish the proof of the corollary, again by \eqref{var:almost} we have
\[\int_{\partial M} 2\alpha_n \vartheta e_0 w e^{-f} dA = 0\]
 this implies $e_0 w = 0$ on $\partial M$. 
\end{proof}
\begin{corollary}\label{Rem:WPM1}
Let $M$ be an $n(\geqslant 3)$-dimensional compact smooth manifold with boundary $\partial M,$ and let  $I^{\alpha_n}_\infty$ be the weighted extended GHY-action on  $\mathscr P(M)$ defined in~\eqref{MGHYF}. The critical points of $I^{\alpha_n}_\infty$ on $\mathcal{C}_1$ are gradient steady solitons on $M$ with totally geodesic boundary satisfying the conditions $e_0 f = 0$ and $e_0 w = 0$ on $\partial M$. 
\end{corollary} 
\begin{proof}
The argument is very similar to the proof of Corollary~\ref{Rem:WPM}. Suppose that it has already been proven that $(g,w, f)$ is a gradient steady soliton. Since
\begin{comment}
 \begin{align}\label{alm:crit}
\int_{\partial M}\Big( 2\alpha_n \vartheta e_0 w - v^{ij}\mathcal A_{ij} - v_{00}(H + e_0 f) \Big)e^{-f} dA = 0,
\end{align}     
\end{comment}
\begin{align}\label{alm:crit}
\int_{\partial M}\Big( 2\alpha_n \vartheta e_0 w - \langle v, \mathcal A\rangle - \langle v, (H + e_0 f) e^\flat_0 \otimes e^\flat_0\rangle \Big)e^{-f} dA = 0,
\end{align}  
for all 
\begin{comment}
$v_{\al\bt}, \vartheta \in C^\infty(M)$
$v \in S^2(M), \vartheta \in C^\infty(M)$
\end{comment}
$(v, \vartheta) \in \mathcal C_1$ from which we obtain that the critical points are
gradient steady solitons on $M$ with totally geodesic boundary satisfying the conditions $e_0 f = 0$ and $e_0 w = 0$ on $\partial M$.
\end{proof}

 Next, we compute the time-derivative of $I_\infty^{\alpha_n}$ under Perelman’s modified extended Ricci flow.

\begin{proposition}\label{alm:mon}
Let $M$ be an $n(\geqslant 3)$-dimensional compact smooth manifold with boundary $\partial M,$ and let $I^{\alpha_n}_\infty$ be as in \eqref{MGHYF}. If $(g(t),w(t))\in{\rm met}(M)\times C^\infty (M),$ $t\in[0,T),$ is a Perelman’s modified extended Ricci flow, then the following equality holds:
\begin{align*}
\dfrac{d}{d t}I_\infty^{\alpha_n} =& 2 \int_{M} \Big(|\Ric + \nabla^2 f -\alpha_n dw \otimes dw|^2 +\alpha_n\big(\Delta w - \langle \nabla w, \nabla f \rangle\big)^2\Big) e^{-f}dV \\
& + 2 \int_{\partial M} \Big( \widehat{\Delta} H - 2 \langle \widehat{\nabla} f,  \widehat{\nabla} H \rangle + \mathcal{A}( \widehat{\nabla} f,  \widehat{\nabla} f) + \mathcal{A}^{ij} \mathcal{A}_{ij} H + \mathcal{A}^{ij}R_{ij} \\
& + 2 R^{0 i}  \widehat{\nabla}_i f -  \widehat{\nabla}_i R^{0i} - \alpha_n \mathcal{A}(\widehat{\nabla} w, \widehat{\nabla} w )\Big)e^{-f} d A.
\end{align*}
In particular, if both $\big(R_{ij} + \nabla_i \nabla_j f - \alpha_n \nabla_i w \nabla_j w\big)|_{\partial M}$ and $\big(R_{i0} + \nabla_i \nabla_0 f\big)|_{\partial M}$ vanish, then the boundary integrand vanishes.
\end{proposition}
\begin{proof}
By~\eqref{grad:form:eq1} and~\eqref{grad:form:eq2} we have $v_{\alpha\beta} = 2(\alpha_n \nabla_\alpha w \nabla_\beta w  - R_{\alpha\beta} - \nabla_\alpha\nabla_\beta f)$ and $\vartheta=\Delta w - \langle \nabla w, \nabla f\rangle,$ respectively. Tracing the equation~\eqref{grad:form:eq1} and using~\eqref{weighted_pr_meas}, we obtain $\frac{v}{2} - h = 0$ on $M$, which allows us to use Proposition~\ref{var:MWGHY} to get
\begin{align*}
\dfrac{d}{d t}I_\infty^{\alpha_n}=& 2 \int_{M} \Big(|{\rm Ric}\, + \nabla^2 f -\alpha_n dw \otimes dw|^2 +\alpha_n\big(\Delta w - \langle \nabla w, \nabla f \rangle\big)^2\Big) e^{-f}dV \\
&  + 2\int_{\partial M} \big( \mathcal{A}^{ij}( R_{ij} + \nabla_i\nabla_j f -\alpha_n \nabla_i w \nabla_j w   ) \big) e^{-f} d A,
\end{align*}
where we used that $H+e_0 f = 0$ and $e_0 w = 0$ on $\partial M$. On the other hand, Lemma~1 in Lott~\cite{John_Lott} guarantees that
\begin{align*}
&\mathcal{A}^{ij}\left(R_{i j}+\nabla_i \nabla_j f\right) e^{-f}-\widehat{\nabla}_i\Big(\left(R^{i 0}+\nabla^i \nabla^0 f\right) e^{-f}\Big) \\
=&\!\Big(\widehat{\Delta} H-2\langle\widehat{\nabla} f, \widehat{\nabla} H\rangle+\mathcal{A}(\widehat{\nabla} f, \widehat{\nabla} f)+\mathcal{A}^{i j} \mathcal{A}_{i j} H+\mathcal{A}^{i j} R_{i j}+2 R^{0 i} \widehat{\nabla}_i f-\widehat{\nabla}_i R^{0 i}\Big)e^{-f},
\end{align*}
where $ \nabla^i \nabla^{0} f = g^{ik}g^{0\alpha} \nabla_k \nabla_\alpha f$. Then
\begin{align}\label{vanishes}
&\mathcal{A}^{ij}(R_{ij} + \nabla_i \nabla_j f - \alpha_n \nabla_i w \nabla_j w) e^{-f} - \widehat{\nabla}_i \Big((R^{i0} + \nabla^i \nabla^{0} f  ) e^{-f}\Big)\\
&=\Big( \widehat{\Delta} H - 2 \langle \widehat{\nabla} f,  \widehat{\nabla} H \rangle + \mathcal{A}( \widehat{\nabla} f,  \widehat{\nabla} f) + \mathcal{A}^{ij} \mathcal{A}_{ij} H + \mathcal{A}^{ij}R_{ij} + 2 R^{0 i}  \widehat{\nabla}_i f -  \widehat{\nabla}_i R^{0i}\nonumber\\
&\quad - \alpha_n \mathcal{A}(\widehat{\nabla} w,\widehat{\nabla} w )\Big)e^{-f},\nonumber
\end{align}
and from Stokes' theorem
\begin{align*}
\int_{\partial M}\widehat{\nabla}_i \Big((R^{i0} + \nabla^i \nabla^{0} f ) e^{-f} \Big)dA = 0,
\end{align*}
which is enough to obtain the first part of the theorem. In particular, if both $\big(R_{ij} + \nabla_i \nabla_j f - \alpha_n \nabla_i w \nabla_j w\big)|_{\partial M}$ and $\big(R_{i0} + \nabla_i \nabla_0 f\big)|_{\partial M}$ vanish, then from equation~\eqref{vanishes} the boundary integrand vanishes.
\end{proof}

In our next result, we establish the evolution equations of the geometric quantities of $\partial M$ under Perelman's modified extended Ricci flow. For its proof, we shall need the following identity. 
\begin{align}\label{Simons_identity}
\widehat{\nabla}_i \widehat{\nabla}_j H &= (\widehat{\Delta} \mathcal{A})_{ij} + \widehat{\nabla}_iR_{j0}  
+\widehat{\nabla}_jR_{i0} -\nabla_0R_{ij} + \mathcal{A}^k\!_i R_{0k0j} +\mathcal{A}^k\!_j R_{0k0i}- \mathcal{A}_{ij}R_{0 0} \nonumber\\
&\quad + 2 \mathcal{A}^{kl}R_{kilj} - HR_{0i0j}- H\mathcal{A}^k\!_i \mathcal{A}_{jk} + \mathcal{A}^{kl}\mathcal{A}_{kl}\mathcal{A}_{ij} + \nabla_0 R_{0i0j}.
\end{align}

Identity~\eqref{Simons_identity} has already been observed by Lott~\cite{John_Lott}. Its proof can be obtained from Simons~\cite{James_Simons} or, alternatively, from Huisken~\cite{Huisken1986}. Indeed, in our notations, Lemma~2.1 in~\cite{Huisken1986} becomes
\begin{align*}
\widehat{\nabla}_i \widehat{\nabla}_j H &= (\widehat{\Delta} \mathcal{A})_{ij} - H\mathcal{A}_{ik}\mathcal{A}^k\!_j+ 
\mathcal{A}^{kl}\mathcal{A}_{kl}\mathcal{A}_{ij}  - HR_{0i0j}  +\mathcal{A}_{ij}R_{0 k 0}^k - \mathcal{A}^k\!_{j} R_{kli}^{l}\\
&\quad - \mathcal{A}^k\!_{i} R_{klj}^{l}  + 2 \mathcal{A}^{kl}R_{kilj} +\nabla_jR_{0ki}^k- \nabla_0 R_{ikj}^k  + \nabla_iR_{0kj}^k.
\end{align*}
Hence, equation~\eqref{Simons_identity} follows from the equality $\nabla_i R_{j0}=\widehat{\nabla}_i R_{j0}-\mathcal{A}_{ij}R_{00}+\mathcal{A}^k\!_{i}R_{jk}.$

\begin{proposition}\label{Gradient_formulate}
Let $M$ be an $n(\geqslant 3)$-dimensional compact smooth manifold with boundary $\partial M.$  If $(g(t),w(t))\in{\rm met}(M)\times C^\infty (M),$ $t\in[0,T),$ is a Perelman’s modified extended Ricci flow,  then the following evolution equations hold on $\partial M$:
\begin{align}
\frac{\partial}{\partial t}g_{ij} &= -(\mathcal{L}_{\widehat{\nabla }f}g)_{ij}-2(R_{ij}-\alpha_n \widehat{\nabla}_i w \widehat{\nabla}_j w)-2H\mathcal{A}_{ij}, \label{flow_with_mean_curv}\\
\frac{\partial}{\partial t}w &=  \widehat{\Delta} w + \nabla_0\nabla_0 w -\mathcal{L}_{\widehat{\nabla }f}w,\label{List_heat_equation}\\
\frac{\partial}{\partial t}\mathcal{A}_{ij} &= (\widehat{\Delta} \mathcal{A})_{ij} - (\mathcal{L}_{\widehat{\nabla }f}\mathcal{A})_{ij} - 
\mathcal{A}^k_i R^l_{klj} - \mathcal{A}^k_j R^l_{kli} +2 \mathcal{A}^{kl} R_{kilj} -2 H \mathcal{A}_{ik}\mathcal{A}^{k}_j \nonumber\\
&\quad + \mathcal{A}^{kl}\mathcal{A}_{kl}\mathcal{A}_{ij} + \nabla_0 R_{0 i 0 j} \label{sec_fund_evolution}
\end{align}
and
\begin{equation}\label{List_mean_evolution}
\dfrac{\partial}{\partial t}H = \widehat{\Delta}H - \langle \widehat{\nabla} f, \widehat{\nabla} H\rangle + 2 \mathcal{A}^{ij}R_{ij} + \mathcal{A}^{ij}\mathcal{A}_{ij}H + \nabla_0 R_{0 0} - 2\alpha_n \mathcal{A}(\widehat{\nabla} w, \widehat{\nabla}w).
\end{equation}
\end{proposition}
\begin{proof} 
The proof is analogous to the one give for Theorem~3 in~\cite{John_Lott}. Nevertheless, we shall present it here for the reader's convenience.

We start by substituting $\nabla_i\nabla_j f = \widehat{\nabla}_i \widehat{\nabla}_i f + H \mathcal{A}_{ij}$ (as $H + e_0f = 0$) 
into the equation~\eqref{grad:form:eq1} to get
\begin{align*}
\frac{\partial}{\partial t}g_{ij} &= - 2(R_{ij} + \widehat{\nabla}_i  \widehat{\nabla}_j f + H \mathcal{A}_{ij} - \alpha_n \widehat{\nabla}_i w \widehat{\nabla}_j w),
\end{align*}
which is equation~\eqref{flow_with_mean_curv}. Likewise, equation~\eqref{List_heat_equation} 
is only a restriction on boundary of the equation~\eqref{grad:form:eq2}, since $e_0 w = 0$.

To prove equation~\eqref{sec_fund_evolution} we first observe that by \eqref{grad:form:eq1}
\begin{align*}
\frac{1}{2}v_{\alpha\beta} = -(R_{\alpha\beta} + \nabla_\alpha\nabla_\beta f-\alpha_n \nabla_\alpha w \nabla_\beta w)
\end{align*}
and
\begin{align*}
\delta \mathcal{A}_{ij}= \frac{1}{2}(\nabla_i v_{j0} + \nabla_j v_{i0} - \nabla_0 v_{ij}) + \frac{1}{2}v_{00}\mathcal{A}_{ij}.
\end{align*}
Since $e_0 w = 0$ implies $v_{00}=R_{00} + \nabla_0 \nabla_0 f$, the previous equation is rewritten as
\begin{align*}
\dfrac{\partial}{\partial t} \mathcal{A}_{ij} &= -\nabla_i(R_{j0} + \nabla_j \nabla_0 f - \alpha_n \nabla_j w \nabla_0 w ) -\nabla_j(R_{i0} + \nabla_i \nabla_0 f- \alpha_n \nabla_i w \nabla_0 w  )\\
&\quad +\nabla_0(R_{ij} + \nabla_i \nabla_j f - \alpha_n \nabla_i w \nabla_j w) -(R_{0 0} + \nabla_0 \nabla_0 f ) \mathcal{A}_{ij} .
\end{align*}
Now we will compute some terms of this equation. The first one of them is
\begin{align*}
\nabla_i \nabla_j \nabla_0 f &= \widehat{\nabla}_i \nabla_j \nabla_0 f - \mathcal{A}_{ij} \nabla_0 \nabla_0 f + \mathcal{A}^k\!_i \nabla_j\nabla_k f.
\end{align*}
Replacing $\nabla_j \nabla_0 f = -\widehat{\nabla}_j H_g + \mathcal{A}^k\!_j \widehat{\nabla}_k f$ (since $H + e_0f = 0$), we obtain
 \begin{align*}
\nabla_i \nabla_j \nabla_0 f  &= -\widehat{\nabla}_i \widehat{\nabla}_j H + \widehat{\nabla}_i \Big( \mathcal{A}^k\!_j \widehat{\nabla}_k f \Big) - \mathcal{A}_{ij} \nabla_0 \nabla_0 f + \mathcal{A}^k\!_i \widehat{\nabla}_j\widehat{\nabla}_k f + H \mathcal{A}^k\!_i \mathcal{A}_{jk}\\
&= -\widehat{\nabla}_i \widehat{\nabla}_j H + \Big(\widehat{\nabla}_i \mathcal{A}^k\!_j\Big) \widehat{\nabla}_k f + \mathcal{A}^k\!_j\widehat{\nabla}_i\widehat{\nabla}_k f - \mathcal{A}_{ij} \nabla_0 \nabla_0 f + \mathcal{A}^k\!_i \widehat{\nabla}_j\widehat{\nabla}_k f\\
 &\quad+ H \mathcal{A}^k\!_i \mathcal{A}_{jk}.
\end{align*}
The second one of them is
\begin{align*}
\nabla_0 \nabla_i \nabla_j f - \nabla_j \nabla_i \nabla_0 f = \nabla_0 \nabla_j \nabla_i f - \nabla_j \nabla_0 \nabla_i f &= -R_{0jki} \widehat{\nabla}^k f - R_{0j0i}\nabla_0 f.
\end{align*}
The third one of them is
\begin{align*}
\nabla_0 \big(\nabla_i w \nabla_j w\big) &= \partial_0(\nabla_i w \nabla_j w) - \langle \nabla w, \nabla_{\partial_0} \partial_i\rangle\nabla_j w -  \nabla_i w \langle \nabla w, \nabla_{\partial_0} \partial_j\rangle\\
&= \nabla_i \nabla_0 w \nabla_j w + \nabla_i w \nabla_j \nabla_0 w.
\end{align*}
By interchanging  $0$ and $j$  we also obtain
\begin{align*}
\nabla_j \Big(\nabla_i w \nabla_0 w\Big) &= \nabla_i \nabla_j w \nabla_0 w + \nabla_i w \nabla_0 \nabla_j w.
\end{align*}
All this implies that
\begin{align*}
\dfrac{\partial}{\partial t} \mathcal{A}_{ij}  &= \widehat{\nabla}_i \widehat{\nabla}_j H - \big(\widehat{\nabla}_i \mathcal{A}_{kj}  - R_{0jik} \big) \widehat{\nabla}^k f- \mathcal{A}^k\!_i \widehat{\nabla}_j\widehat{\nabla}_k f - \mathcal{A}^k\!_j\widehat{\nabla}_i\widehat{\nabla}_k f + R_{0i0j}H \\
&\quad   -\nabla_iR_{j0}  -\nabla_jR_{i0} +\nabla_0R_{ij} - \mathcal{A}_{ij}R_{0 0}   - H \mathcal{A}^k\!_i \mathcal{A}_{jk}.
\end{align*}
Using Codazzi-Mainardi equation $R_{0jik}=\widehat{\nabla}_i \mathcal{A}_{jk} - \widehat{\nabla}_k \mathcal{A}_{ij}$ one has
\begin{align*}
\dfrac{\partial}{\partial t} \mathcal{A}_{ij} &= \widehat{\nabla}_i \widehat{\nabla}_j H - \Big(\widehat{\nabla}_k \mathcal{A}_{ij}\Big) \widehat{\nabla}^k f- \mathcal{A}^k\!_i \widehat{\nabla}_j\widehat{\nabla}_k f - \mathcal{A}^k\!_j\widehat{\nabla}_i\widehat{\nabla}_k f  + R_{0j0i}H \\
&\quad  -\nabla_iR_{j0}  -\nabla_jR_{i0} +\nabla_0R_{ij} - \mathcal{A}_{ij}R_{0 0}  - H \mathcal{A}^k\!_i \mathcal{A}_{jk}\\
&= \widehat{\nabla}_i \widehat{\nabla}_j H - \Big(\mathcal{L}_{\widehat{\nabla}f} \mathcal{A}\Big)_{ij}  -\nabla_iR_{j0}  -\nabla_jR_{i0} +\nabla_0R_{ij} - \mathcal{A}_{ij}R_{0 0}  + R_{0i0j}H \\
&\quad - H \mathcal{A}^k\!_i \mathcal{A}_{jk}.
\end{align*}
From Simons' identity~\eqref{Simons_identity} we get
\begin{align*}
\dfrac{\partial}{\partial t} \mathcal{A}_{ij}  &=(\widehat{\Delta} \mathcal{A})_{ij} - \Big(\mathcal{L}_{\widehat{\nabla}f} \mathcal{A}\Big)_{ij}  -(\nabla_iR_{j0} -\widehat{\nabla}_iR_{j0})  -(\nabla_jR_{i0} - \widehat{\nabla}_jR_{i0})  -2 \mathcal{A}_{ij}R_{0 0} \\
&\quad + \mathcal{A}^k\!_i R_{0k0j} + \mathcal{A}^k\!_j R_{0k0i} + 2 \mathcal{A}^{kl}R_{kilj}  - 2 H \mathcal{A}^k\!_i \mathcal{A}_{jk} + \mathcal{A}^{kl}\mathcal{A}_{kl}\mathcal{A}_{ij}\\
&\quad  + \nabla_0 R_{0i0j}.
\end{align*}
As $\nabla_i R_{j0} = \widehat{\nabla}_i R_{j0}  - \mathcal{A}_{ij}R_{00} + \mathcal{A}^k\!_{i}R_{jk}$ we conclude that
\begin{align*}
\dfrac{\partial}{\partial t} \mathcal{A}_{ij} &=(\widehat{\Delta} \mathcal{A})_{ij} - \Big(\mathcal{L}_{\widehat{\nabla}f} \mathcal{A}\Big)_{ij}  -\mathcal{A}^k\!_i R^l\!_{klj}  -\mathcal{A}^k\!_j R^l\!_{kli}  + 2 \mathcal{A}^{kl}R_{kilj} - 2 H \mathcal{A}^k\!_i \mathcal{A}_{jk} \\
&\quad + \mathcal{A}^{kl}\mathcal{A}_{kl}\mathcal{A}_{ij} + \nabla_0 R_{0i0j}.
\end{align*}
For finishing our proof, we show equation~\eqref{List_mean_evolution}. For it, note that
\begin{align*}
\delta H = -v_{ij}\mathcal{A}^{ij} + g^{ij}\delta \mathcal{A}_{ij}
\end{align*}
and
\begin{align*}
g^{ij}(\mathcal{L}_{\widehat{\nabla}f} \mathcal{A})_{ij} -2 \mathcal{A}^{ij}\widehat{\nabla}_i\widehat{\nabla}_j f = \widehat{\nabla}_{\widehat{\nabla} f} (g^{ij} \mathcal{A}_{ij})= \langle \widehat{\nabla} f, \widehat{\nabla} H\rangle.
\end{align*}
So,
\begin{align*}
\dfrac{\partial}{\partial t}H 	&=2(R_{ij} + \widehat{\nabla}_i\widehat{\nabla}_j f + H \mathcal{A}_{ij}) \mathcal{A}^{ij} + g^{ij} \Big((\widehat{\Delta} \mathcal{A})_{ij} - \big(\mathcal{L}_{\widehat{\nabla}f} \mathcal{A}\big)_{ij}  -\mathcal{A}^k_i R^l_{klj}   \\
&\quad -\mathcal{A}^k_j R^l_{kli}  + 2 \mathcal{A}^{kl}R_{kilj}  - 2 H \mathcal{A}^k_i \mathcal{A}_{jk} + \mathcal{A}^{kl}\mathcal{A}_{kl}\mathcal{A}_{ij} + \nabla_0 R_{0i0j}  \Big) \\
&\quad - 2\alpha_n \mathcal{A}(\widehat{\nabla} w, \widehat{\nabla} w )\\
&=2\mathcal{A}^{ij}R_{ij}  + 2H \mathcal{A}^{ij}\mathcal{A}_{ij}  + \widehat{\Delta} H     -\Big(g^{ij}\big(\mathcal{L}_{\widehat{\nabla}f} \mathcal{A}\big)_{ij} -2 \mathcal{A}^{ij}\widehat{\nabla}_i\widehat{\nabla}_j f\Big) - 2  \mathcal{A}^{kj} \mathcal{A}_{jk} H   \\
&\quad + \mathcal{A}^{kl}\mathcal{A}_{kl}H + \nabla_0 R_{00} - 2\alpha_n \mathcal{A}(\widehat{\nabla} w, \widehat{\nabla} w )\\
&= \widehat{\Delta}H - \langle \widehat{\nabla} f, \widehat{\nabla} H\rangle + 2 \mathcal{A}^{ij}R_{ij} + \mathcal{A}^{ij}\mathcal{A}_{ij}H + \nabla_0 R_{0 0} - 2\alpha_n \mathcal{A}(\widehat{\nabla} w, \widehat{\nabla}w).
\end{align*}
This finishes the proof.
\end{proof}

 As a consequence of Proposition~\ref{Gradient_formulate}, we have the following refinement of the formula obtained in Proposition~\ref{alm:mon}.  
\begin{corollary}\label{key_prop_monotonicity}
Let $M$ be an $n(\geqslant 3)$-dimensional compact smooth manifold with boundary $\partial M$, and let $I^{\alpha_n}_\infty$ be as in~\eqref{MGHYF}. 
 If $(g(t),w(t))\in{\rm met}(M)\times C^\infty (M),$ $t\in[0,T),$
is a Perelman’s modified extended Ricci flow,  then the following identity holds:
\begin{align*}
\dfrac{d}{d t}I_\infty^{\alpha_n} &=2 \int_{M} \Big(\left|\Ric + \nabla^2 f -\alpha_n dw \otimes dw\right|^2 + \alpha_n\big(\Delta w - \langle \nabla w, \nabla f \rangle\big)^2 \Big) e^{-f}dV \\
&\quad + 2 \int_{\partial M} \Big(\dfrac{\partial H}{\partial t}  - \langle \widehat{\nabla} f, \widehat{\nabla} H\rangle + \mathcal{A}(\widehat{\nabla} f, \widehat{\nabla} f) + 2 R^{0 i}\widehat{\nabla}_i f- \dfrac{1}{2} \nabla_0 R - HR_{00}\\
&\quad +  \alpha_n \mathcal{A} (\widehat{\nabla} w, \widehat{\nabla} w)\Big)e^{-f} d A.
\end{align*}
In particular, if both $\big(R_{ij} + \nabla_i \nabla_j f - \alpha_n \nabla_i w \nabla_j w\big)|_{\partial M}$ and $\big(R_{i0} + \nabla_i \nabla_0 f\big)|_{\partial M}$ vanish, then the boundary integrand vanishes.
\end{corollary}
\begin{proof}
From equation~\eqref{List_mean_evolution} of Proposition~\ref{Gradient_formulate}, the boundary integrand term of Proposition~\ref{alm:mon} can be rewritten as
\begin{align*}
&\widehat{\Delta} H - 2 \langle \widehat{\nabla} f,  \widehat{\nabla} H \rangle + \mathcal{A}( \widehat{\nabla} f,  \widehat{\nabla} f) + \mathcal{A}^{ij} \mathcal{A}_{ij} H + \mathcal{A}^{ij}R_{ij} + 2 R^{0 i}  \widehat{\nabla}_i f -  \widehat{\nabla}_i R^{0i}\\
&\quad - \alpha_n \mathcal{A}(\widehat{\nabla} w, \widehat{\nabla} w )\\
& = \dfrac{\partial H}{\partial t} -  \langle \widehat{\nabla} f,  \widehat{\nabla} H \rangle + \mathcal{A}( \widehat{\nabla} f,  \widehat{\nabla} f)  - \mathcal{A}^{ij}R_{ij} + 2 R^{0 i}  \widehat{\nabla}_i f -  \widehat{\nabla}_i R^{0i}  - \nabla_0 R_{00} \\
&\quad  + \alpha_n \mathcal{A}(\widehat{\nabla} w, \widehat{\nabla} w ).
\end{align*}
Contracted Bianchi Identity and the fact that $\nabla_i R_{j0} = \widehat{\nabla}_i R_{j0} - \mathcal{A}_{ij}R_{00} + \mathcal{A}^k\!_{i}R_{jk}$ imply
\begin{align*}
\dfrac{1}{2} \nabla_0 R = \nabla_i R^{i0} + \nabla_0 R_{00} = \widehat{\nabla}_i R^{i0}  - HR_{00} + \mathcal{A}^{ij}R_{ij} + \nabla_0 R_{00}.
\end{align*}
The main result of the corollary follows from these two latter equations. If, in addition, both $R_{ij} + \nabla_i \nabla_j f - \alpha_n \nabla_i w \nabla_j w$ and $R_{i0} + \nabla_i \nabla_0 f$ vanish on $\partial M$, then by Proposition~\ref{alm:mon} the integrand of $\partial M$, namely
 \begin{align*}
\dfrac{\partial H}{\partial t}  - \langle \widehat{\nabla} f, \widehat{\nabla} H\rangle + \mathcal{A}(\widehat{\nabla} f, \widehat{\nabla} f) + 2 R^{0 i}\widehat{\nabla}_i f- \dfrac{1}{2} \nabla_0 R - HR_{00} +  \alpha_n \mathcal{A} (\widehat{\nabla} w, \widehat{\nabla} w)
\end{align*}
vanishes.
\end{proof}

\section{Hypersurfaces in an extended Ricci flow background} \label{se: Ricci}

We give in this section the proofs of our main results. To prove Theorem~\ref{principal_theorem}, we shall need the following

\begin{proposition}\label{mean_curv_flow_in_a_mod}
Let $M$ be an $n(\geqslant 3)$-dimensional smooth manifold. Suppose $\mathscr F:=\{\Sigma_t\,;\, t\in[0,T)\}$ is a mean curvature flow in the $(g(t),w(t))$-extended Ricci flow background on $M$ which satisfies $e_0w=0$ on $\Sigma_0,$ where $e_0$ is the  unit normal vector field on $\Sigma_0.$ Then, the following evolution equations hold: 
\begin{align}
\frac{\partial}{\partial t}g_{ij} &= - 2(R_{ij}  - \alpha_n \widehat{\nabla}_i w \widehat{\nabla}_j w)-2H\mathcal{A}_{ij},\label{List:together:mean:curv}\\
\frac{\partial}{\partial t}w &=  \widehat{\Delta} w + \nabla_0\nabla_0 w,\label{List:secondEq}\\
\frac{\partial}{\partial t}\mathcal{A}_{ij} &=
(\widehat{\Delta} \mathcal{A})_{ij}-\mathcal{A}^k\!_i R^l_{klj}-\mathcal{A}^k\!_j R^l_{kli}+2\mathcal{A}^{kl} R_{kilj}
-2H  \mathcal{A}_{ik}\mathcal{A}^k\!_j  \label{sec_fund_variational}\\
&\quad + \mathcal{A}^{kl}\mathcal{A}_{kl}\mathcal{A}_{ij} + \nabla_0 R_{0 i 0 j} \nonumber
\end{align}
and
\begin{align}\label{mean_curv_background}
\dfrac{\partial}{\partial t}H = \widehat{\Delta}H + 2 \mathcal{A}^{ij}R_{ij} + \mathcal{A}^{ij}\mathcal{A}_{ij}H
+ \nabla_0 R_{0 0} -2\alpha_n \mathcal{A}(\widehat{\nabla}w, \widehat{\nabla}w).
\end{align}
\end{proposition}
\begin{proof}
In this proof, we follow ~\cite[Proposition~4]{John_Lott} closely. First, assume $\Sigma_t=\partial X_t$ with each $X_t$ compact. Given a time interval $[a, b]$, we can find a positive solution $u=
e^{-f}$ on $\bigcup_{t\in [a,b]}(X_t \times \{t\}) \subset M\times [a, b]$ of the conjugate heat equation
 \begin{align}\label{assump_principal_thm}
\frac{\partial}{\partial t}u = - \Delta u + R u - \alpha_n |\nabla w|^2u
\end{align}
satisfying the boundary condition $e_0 u = Hu$ on $\Sigma_0=\partial X_a$, by solving it backwards
in time from $t = b$. (Choosing diffeomorphisms from $\{ X_t \}$ to $X_a$, we can reduce the problem
of solving~\eqref{assump_principal_thm} to a parabolic equation on a fixed domain with $e_0 w = 0$ on $\Sigma_0).$

Now, let $\{\phi_t\}_{t \in [a, b]}$ be the one-parameter family of
diffeomorphisms generated by $\{-\nabla_{g(t)} f(t)\}_{t \in [a, b]}$,
with $\phi_a = \Id$. Then $\phi_t(X_a) = X_t$ for all $t.$
By setting $\widetilde{g}(t) =\phi_t^*g(t)$, $\widetilde{w}(t) =\phi_t^*w(t)$
and $\widetilde{f}(t) =\phi_t^*f(t)$
we have that $\widetilde{g}(t)$, $\widetilde{w}(t)$ and $\widetilde{f}(t)$ are defined on $X_a$.
We claim that
\begin{align*}
\left\{
\begin{array}{lcl}
\frac{\partial }{\partial t} \widetilde{g}_{\alpha \beta}=
-2 (\widetilde{R}_{\alpha\beta} +  \widetilde{\nabla}_\alpha\widetilde{\nabla}_\beta \widetilde{f}
- \alpha_n \widetilde{\nabla}_\alpha \widetilde{w}  \widetilde{\nabla}_\beta \widetilde{w}),\\
\frac{\partial }{\partial t} \widetilde{w} = \Delta_{\widetilde{g}} \widetilde{w}
- \langle \widetilde{\nabla} \widetilde{w} , \widetilde{\nabla}\widetilde{f} \rangle_{\widetilde{g}}
\end{array}
\right.
\end{align*}
and
\begin{align*}
\frac{\partial}{\partial t}\widetilde{f} =
- \Delta_{\widetilde{g}} \widetilde{f} - R_{\widetilde{g}}  + \alpha_n |\widetilde{\nabla} \widetilde{w}|_{\widetilde{g}}^2
\end{align*}
in $ X_a \times [a,b]$ with $e_0f + H = 0$ and $e_0 w = 0$ on $ \partial X_a= \Sigma_0.$
Indeed,
\begin{align*}
\frac{\partial }{\partial t}\widetilde{g}_{\alpha \beta} &= \phi^*_t \Big(\frac{\partial }{\partial t} g_{\alpha\beta}\Big) + \phi^*_t\Big(\mathcal{L}_{\frac{d}{dt}\phi_t}g\Big)_{\alpha\beta}\\
&=\phi^*_t \Big( -2(R_{\alpha\beta}  - \alpha_n \nabla_\alpha w \nabla_\beta w)\Big) - \phi^*_t\Big(\mathcal{L}_{\big(\nabla_{g(t)} f(t)\big)}g\Big)_{\alpha\beta}\\
&= -2 (\widetilde{R}_{\alpha\beta} +  \widetilde{\nabla}_\alpha\widetilde{\nabla}_\beta \widetilde{f} - \alpha_n \widetilde{\nabla}_\alpha \widetilde{w}  \widetilde{\nabla}_\beta \widetilde{w}).
\end{align*}
For the second item, we have
\begin{align*}
\frac{\partial }{\partial t}\widetilde{w} &= \phi^*_t \Big(\frac{\partial }{\partial t} w\Big) + \phi^*_t\mathcal{L}_{\frac{d}{dt} \phi_t}w \\
&= \phi^*_t \Big(\Delta w\Big) - \phi^*_t\mathcal{L}_{\big(\nabla_{g(t)} f(t)\big)} w \\
&=\Delta_{\widetilde{g}} \widetilde{w}  - \langle \widetilde{\nabla} \widetilde{w} , \widetilde{\nabla}\widetilde{f} \rangle_{\widetilde{g}}.
\end{align*}
Now, we use that $\Delta u = (|\nabla f|^2 - \Delta f) e^{-f}$ and~\eqref{assump_principal_thm} to obtain
\begin{align*}
\frac{\partial }{\partial t}\widetilde{f} &= \phi^*_t \Big(\frac{\partial }{\partial t} f\Big) + \phi^*_t\mathcal{L}_{\frac{d}{dt} \phi_t}f \\
&= \phi^*_t \Big( |\nabla f|^2 - \Delta f - R  + \alpha_n |\nabla w|^2\Big) - \phi^*_t\mathcal{L}_{\big(\nabla_{g(t)} f(t)\big)} f \\
&=  - \Delta_{\widetilde{g}} \widetilde{f} - R_{\widetilde{g}}  + \alpha_n |\widetilde{\nabla} \widetilde{w}|_{\widetilde{g}}^2 .
\end{align*}
The boundary conditions follow from the fact that $e_0u=Hu$ and $e_0 w = 0$ on $\Sigma_0.$  Thus, $(\widetilde{g}(t), \widetilde{w}(t))$
evolves by Perelman’s modified extended Ricci flow in $X_a \times [a,b]$, and then we are in a position to use Proposition~\ref{Gradient_formulate}
for the compact manifold $X_a$ with boundary $\partial X_a$. So, from equation~\eqref{flow_with_mean_curv} we have on $\Sigma_t$
\begin{align*}
\frac{\partial }{\partial t} g_{ij}  &=  \frac{\partial }{\partial t} \Big((\phi^*_t)^{-1} \phi^*_t g_{ij}\Big) = \frac{\partial }{\partial t}\Big((\phi^*_t)^{-1} \widetilde{g}_{ij}\Big)
=(\phi^*_t)^{-1} \Big(\frac{\partial }{\partial t} \widetilde{g}_{ij}  + \Big(\mathcal{L}_{\frac{d}{dt}\phi_t^{-1}}\widetilde{g}\Big)_{ij} \Big)\\
&= - 2(R_{ij}  - \alpha_n \widehat{\nabla}_i w \widehat{\nabla}_j w) - 2 H \mathcal{A}_{ij},
\end{align*}
 which is~\eqref{List:together:mean:curv}. Likewise, from equation~\eqref{List_heat_equation} one has
\begin{align*}
\frac{\partial }{\partial t} w &= (\phi^*_t)^{-1} \Big( \frac{\partial }{\partial t} \widetilde{w} +\mathcal{L}_{\frac{d}{dt}\phi_t^{-1}}\widetilde{w} \Big)
=\widehat{\Delta} w + \nabla_0 \nabla_0 w,
\end{align*}
which is~\eqref{List:secondEq}. Next, equation~\eqref{sec_fund_evolution} implies
\begin{align*}
\frac{\partial }{\partial t}\mathcal{A}_{ij} &=
(\phi^*_t)^{-1}\Big(\frac{\partial}{\partial t}\widetilde{\mathcal{A}}_{ij}+ \Big(\mathcal{L}_{\frac{d}{dt}\phi_t^{-1}}\widetilde{\mathcal{A}}\Big)_{ij}\Big) \\
&= (\widehat{\Delta} \mathcal{A})_{ij}  - \mathcal{A}^k\!_i R^l_{klj} - \mathcal{A}^k\!_j R^l_{kli}
+2 \mathcal{A}^{kl} R_{kilj} -2H  \mathcal{A}_{ik}\mathcal{A}^k\!_j+\mathcal{A}^{kl}\mathcal{A}_{kl}\mathcal{A}_{ij}\\
&\quad+ \nabla_0 R_{0 i 0 j}.
\end{align*}
For finishing, from equation~\eqref{List_mean_evolution} we get
\begin{align*}
\frac{\partial }{\partial t} H &= (\phi^*_t)^{-1} \Big(\frac{\partial }{\partial t}H_{\widetilde{g}} + \mathcal{L}_{\frac{d}{dt} \phi_t^{-1}}H_{\widetilde{g}} \Big) \\
&= \widehat{\Delta}H + 2 \mathcal{A}^{ij}R_{ij} + \mathcal{A}^{ij}\mathcal{A}_{ij}H + \nabla_0 R_{0 0} -2\alpha_n \mathcal{A}(\widehat{\nabla}w, \widehat{\nabla}w).
\end{align*}
This finishes the proof.
\end{proof}

\begin{remark}\label{without}
We  point out that equations~\eqref{List:together:mean:curv} and~\eqref{List:secondEq} hold  regardless the assumption $e_0 w=0$ on $\Sigma_0$.
\end{remark}

\begin{remark} \label{rem: recover2}
If $M$ is the Euclidean space with its standard metric $g_0$, $g(t) = g_0$ and $w(t) = w$ is a constant,
then Eqs.~\eqref{List:together:mean:curv}, \eqref{sec_fund_variational} and \eqref{mean_curv_background}
are the same as in~\cite[Lem.~3.2, Thm.~3.4 and Cor.~3.5]{Huisken}, see also Mantegazza~\cite[Sect.~2.3]{Mantegazza}.
\end{remark}

\begin{proof}[\bf Proof of Theorem~\ref{principal_theorem}]
The hypotheses on $\{{\partial M}_t\,;\,t\in[0,T)\}$ and on $u$ allow us to use $\widetilde{g}(t)$, $\widetilde{w}(t)$ and $\widetilde{f}(t)$ on $M$ as in the proof of Proposition~\ref{mean_curv_flow_in_a_mod}.  In this way, the result follows immediately from Corollary~\ref{key_prop_monotonicity} and the fact that the identity 
$$\frac{\partial}{\partial t} H_{\widetilde{g}}=
\dfrac{\partial}{\partial t}H_g-\langle \widehat{\nabla} f,\widehat{\nabla} H\rangle$$ 
holds on $\partial M_t$ for all $t\in[0,T).$ 
\end{proof}

\begin{remark} \label{rem: recover3}
 As we pointed out in the introduction, our Theorem~\ref{principal_theorem} extends Theorem~1 in~\cite{John_Lott}. Also, when $M$ is compact without boundary, it coincides with~\cite[Lem.~3.4]{Bernhard_List}.
\end{remark}

\section{Extension of Hamilton's differential Harnack expression}

Here, we will see as the boundary integrand term of the time-derivative of weighted extended GHY-action provides an extension of Hamilton’s differential Harnack expression for mean curvature flow in Euclidean space to the more general context of mean curvature flow in an extended Ricci flow background.

 Let $\mathscr F:=\{\Sigma_t\}$ be a family of mean curvature solitons in the $(\overline{g},\overline{w})$-extended Ricci flow background.
Then, the equations for the steady case 
$$\overline{R}_{ij} + \overline{\nabla}_i \overline{\nabla}_j\overline{f} - \alpha_n \overline{\nabla}_i \overline{w} \overline{\nabla}_j \overline{w} 
= 0\quad \text{and}\quad  \overline{R}_{i0} + \overline{\nabla}_i \overline{\nabla}_0\overline{f} - 
\alpha_n \overline{\nabla}_i \overline{w} \overline{\nabla}_0 \overline{w} = 0$$  on $\Sigma_t$ become
\begin{align}\label{soliton_restricted}
\overline{R}_{ij} + \widehat{\nabla}_i \widehat{\nabla}_j\overline{f} 
+ H_{\overline{g}} \mathcal{A}_{ij}- \alpha_n\widehat{\nabla}_i \overline{w}\widehat{\nabla}_j\overline{w}=0,
\end{align}
and
\begin{align}\label{soliton_restricted2}
\overline{R}_{i0} - \widehat{\nabla}_i H_{\overline{g}} + \mathcal{A}^k\!_i \widehat{\nabla}_k \overline{f} -\alpha_n e_0 \overline{w}\widehat{\nabla}_i \overline{w}= 0 .
\end{align}

\begin{example}
For instance, consider $M = \mathbb{R}^n$, $\overline{g}(t)= \delta_{\alpha \beta}$ and $\overline{w}(t) = w$ constant, and let $L$ be a linear function on $\mathbb{R}^n$.  Defining $\overline{f} = L + t |\nabla L|^2$, we have that $\overline f$ satisfies~\eqref{self-solution}. Changing $\overline{f}$ to $-f$, equations~\eqref{soliton_restricted} and~\eqref{soliton_restricted2}  then become
\begin{align*}
\widehat{\nabla}_i \widehat{\nabla}_j f-H \mathcal{A}_{ij}=0 \quad \hbox{and}\quad \widehat{\nabla}_i H + \mathcal{A}^k\!_i \widehat{\nabla}_k f = 0,
\end{align*}
respectively, which appear in~\cite[p.~219]{Hamilton} as equations for a translating soliton.
\end{example}

Consider a bounded domain $\Omega$  with smooth boundary $\partial\Omega:=\Sigma$ in Euclidean space $\mathbb{R}^n$, and   take a solution $u=e^{-f}$  to the  conjugate heat
equation~\eqref{back:heat:eq} in $\Omega\times[0,T)$
with $e_0 u=Hu$ on $\Sigma$. If $\mathscr{F}:=\{\Sigma_t\,;\, t\in[0,T)\}$ is a
mean curvature flow in a $(g(t),w(t))$-extended Ricci flow background
with $g(t)$ Ricci flat and $e_0 w = 0$ on $\Sigma$, then the boundary
integrand in Theorem~\ref{principal_theorem} becomes
\begin{align}\label{type:Harnack}
&\mathcal{Z}(V) + \alpha_n \mathcal{A}(\widehat{\nabla} w, \widehat{\nabla} w),
\end{align}
where $V=-\widehat{\nabla} f$ and $\mathcal{Z}(V):=\frac{\partial H}{\partial t}
+ 2\langle V, \widehat{\nabla} H\rangle + \mathcal{A}(V, V)$
is Hamilton’s differential Harnack expression for the case of mean
curvature flow in Euclidean space, which vanishes in the particular case
$\mathscr{F}$ is a translating soliton (cf. ~\cite[Def.~4.1 and Lem.~3.2]{Hamilton}).

The next result suggests an extension $\mathcal{Z}^{\alpha_n}_{\overline{g},\overline{w}}$ of $\mathcal{Z}$ for the more general case of mean curvature flow in an extended Ricci flow background, whose characterization of nullity should be on the steady case. For this, we observe that, if $(\overline{g}(t), \overline{w}(t))$ is a gradient steady soliton on a smooth manifold $M$ with potential function $\overline{f}$, and $\Sigma$ is a mean curvature soliton at $t=0$, then its ensuing mean curvature flow $\{\Sigma_t \}$ consists of mean curvature solitons, and $\{\Sigma_t \}$ differs from $\{\psi_t(\Sigma)\}$ by hypersurface diffeomorphisms. In Section~\ref{sect:grad:sol}, we give a more general description that includes the shrinking and expanding soliton cases.

\begin{corollary}\label{type-Harnack-background}
Let $M$ be an $n(\geqslant 3)$-dimensional smooth manifold and $(\overline{g}(t), \overline{w}(t))$ a gradient steady soliton on $M\times[0,T)$ with potential function $\overline{f}.$ Assume that  $\mathscr F:=\{\Sigma_t\,;\, t\in[0,T)\}$ is a mean curvature flow in the $(\overline g,\overline w)$-extended Ricci flow background which satisfies $H + e_0 f = 0$ and $e_0w=0$ on $\Sigma_0,$ where $e_0$ is the  unit normal vector field on $\Sigma_0.$ Under these conditions, the  identity
\begin{align*}
\mathcal{Z}(-\widehat{\nabla}_{\overline{g}}\overline{f})+
 2 \overline{R}^{0 i}\widehat{\nabla}_i \overline{f}- \dfrac{1}{2} \overline{\nabla}_0 \overline{R}
 - H_{\overline{g}}\overline{R}_{00}+  
 \alpha_n \mathcal{A}(\widehat{\nabla}_{\overline{g}} \overline{w}, \widehat{\nabla}_{\overline{g}} \overline{w}) = 0
\end{align*}
holds for all $t\in[0,T)$,  where $\mathcal{A}$ and $\widehat{\nabla}_{\overline{g}}$ are as in Theorem \ref{principal_theorem}.
\end{corollary}

\begin{proof}
If $(\overline{g}(t), \overline{w}(t))$ is a gradient steady soliton on $M \times [0,T)$, then the positive function $u = e^{-\overline{f}(t)}$ on $\bigcup_{t\in [0,T)}(X_t \times \{t\}) \subset M\times [0, T)$ satisfies the conjugated heat equation~\eqref{assump_principal_thm} with $e_0 u = Hu$ and $e_0 w = 0$ on $\partial X_0=\Sigma_0$, where the boundary conditions follows from the assumptions on $\Sigma_0$. To see this, first observe that $\Delta_{\overline{g}} u = (|\nabla_{\overline{g}} \overline{f}|_{\overline{g}}^2 - \Delta_{\overline{g}} \overline{f}) u$. Now taking traces in the first equation of~\eqref{mod_grad_Ricci_soliton} and using~\eqref{self-solution}, we obtain
 \begin{align*}
\frac{\partial}{\partial t} u = -u |\nabla_{\overline{g}} \overline{f}|_{\overline{g}}^2 = -\Delta_{\overline{g}} u+ R_{\overline{g}} u - \alpha_n |\nabla_{\overline{g}} \overline{w}|_{\overline{g}}^2u.
 \end{align*}
Thus, we can define $\widetilde{g}(t)$, $\widetilde{w}(t)$ and $\widetilde{f}(t)$ on $X_0$ as in the proof of Proposition~\ref{mean_curv_flow_in_a_mod}, so that $(\widetilde{g}(t), \widetilde{w}(t))$ evolves by Perelman’s modified extended Ricci flow on $X_0 \times [0,T)$. Besides, again we use that $(\overline{g}(t), \overline{w}(t))$ is a gradient steady soliton and that $e_0 w = 0$ on $\Sigma_0$, to get
\begin{align*}
 \big(\widetilde{R}_{ij} + \widetilde{\nabla}_i \widetilde{\nabla}_j \widetilde{f} - \alpha_n \widetilde{\nabla}_i \widetilde{w} \widetilde{\nabla}_j \widetilde{w}\big)|_{\Sigma_0}=0 \quad\hbox{and}\quad  \big(\widetilde{R}_{i0} + \widetilde{\nabla}_i \widetilde{\nabla}_0 \widetilde{f}\big)|_{\Sigma_0}=0.
\end{align*}
As in the proof of Theorem~\ref{principal_theorem}, the result of the corollary follows from Corollary~\ref{key_prop_monotonicity} 
and the identity 
$$\dfrac{\partial}{\partial t} H_{\widetilde{g}} = 
\dfrac{\partial}{\partial t} H_{\overline{g}} -
\langle \widehat{\nabla}_{\overline{g}} \overline{f}, \widehat{\nabla}_{\overline{g}} H_{\overline{g}}\rangle_{\overline{g}}.$$
This completes the proof.
\end{proof}

\begin{remark} \label{rem: recover4}
Suppose $M = \mathbb{R}^n$, $\overline{g}_{\alpha\beta}(t) = \delta_{\alpha\beta}$ and $\overline{w}(t) = w$ constant. 
Let $L$ be a linear function on $\mathbb{R}^n$ and define 
$\overline{f} = L + t |\nabla L|^2$. Letting $V(t) = 
-\widehat{\nabla} \overline{f}$, Corollary~\ref{type-Harnack-background}  coincides with~\cite[Lem.~3.2]{Hamilton}.
\end{remark}

\section{Characterization of mean curvature solitons}\label{sect:grad:sol}
Special solutions of extended Ricci flow come from gradient solitons. Here we describe such solutions on a background geometry, we follow as in Ph.D. thesis by List~\cite[Sect.~2.2]{Bernhard_List_Thesis} in the line of Lott and Kleiner~\cite[Appx.~C]{Kleiner_Lott}. We start recalling some aspects of the theory of extended Ricci solitons.

A gradient soliton to the extended Ricci flow is, by definition, a self-similar solution $(\overline g(t),\overline w(t))$ of~\eqref{ListEq} given by
\begin{subequations}
\begin{empheq}[left=\empheqlbrace]{align}
\overline g(t) &= \sigma(t)\psi_{t}^*g,\label{self-similiar:solution1}\\
\overline w(t) &= \psi_t^* w\label{self-similiar:solution2},
\end{empheq}
\end{subequations}
for some initial value $(g,w)$, where $\psi_t$ is a  smooth one-parameter  family of diffeomorphisms of $M$ generated from the flow of $\nabla_g f/\sigma(t)$ computed on $g$, for some $f \in C^\infty(M)$, and $\sigma$  is a smooth positive function on $t.$ Gradient solitons to the extended Ricci flow are obtained as follows.

\begin{proposition}\label{self:similar:solution:alt}
Let $M$ be an $n(\geqslant 3)$-dimensional smooth manifold. Suppose there exists a triple $(g, w, f)$ satisfying
\begin{subequations}
\begin{empheq}[left=\empheqlbrace]{align}
&\Ric_{g} + \nabla^2_{g}f - \alpha_n d w  \otimes d w = \lambda g,\label{self-similar:sol1:alt}\\
&\Delta_{g} w = \langle \nabla_{g} f, \nabla_{g} w\rangle_{g}\label{self-similar:sol2:alt}
\end{empheq}
\end{subequations}
for some $\lambda \in \mathbb R$ and $\alpha_n = (n - 1)/(n - 2).$ Take $\psi_t$ the one-parameter family of diffeomorphisms generated by $Y_t = \frac{\nabla_g f}{\sigma(t)}$, with $\psi_{0} = \Id$ and $\sigma(t) = 1 -2\lambda t > 0,$ where $t\in (-\infty,\frac{1}{2\lambda})$, for $\lambda>0$; $t\in\mathbb R$, for $\lambda=0$; and $t\in(\frac{1}{2\lambda},+\infty)$, for $\lambda<0$.  Then, $(\sigma(t) \psi_t^* g  ,  \psi_t^* w)$ is a gradient soliton to the extended Ricci flow on $M.$ 
\end{proposition}
\begin{proof}
Setting $\overline g(t) = \sigma(t) \psi_t^* g$ and $\overline w(t) = \psi_t^* w,$ one has
\begin{align*}
 \frac{\partial}{\partial t} \overline g(t) &=  \sigma'(t) \psi_t^*g + \sigma(t)\psi^*_t(\mathcal L_{Y_t}g)= \psi^*_t( -2\lambda g + 2\nabla^2_{g} f) =-2\psi^*_t( \lambda g - \nabla^2_{g} f) \\
 &=-2 \psi^*_t(\Ric_g - \alpha_n d w \otimes d w  ) = -2 \Ric_{\overline g(t)} + 2\alpha_n d \overline w(t) \otimes d \overline w(t)  
\end{align*}
and
\begin{align*}
\frac{\partial}{\partial t} \overline w(t) &=  \psi^*_t(\mathcal L_{Y_t}w) = \frac{1}{\sigma(t)} \psi^*_t  \mathcal L_{\nabla_{g}f} w = \frac{1}{\sigma(t)} \psi^*_t \langle\nabla_{g} f,\nabla_{g} w\rangle_g\\
&= \frac{1}{\sigma(t)}\psi^*_t \Delta_{g}w=\Delta_{\overline g(t)}\overline w(t).
\end{align*} 
This completes the proof.
\end{proof}

Note that a gradient soliton to the extended Ricci flow on $M$ is characterized by means Proposition~\ref{self:similar:solution:alt}, i.e., we can say that a gradient soliton on $M$ is a triple $(g,f,w)$ satisfying \eqref{self-similar:sol1:alt} (as well \eqref{self-similar:sol2:alt}). It is steady if $\lambda = 0$, shrinking if $\lambda > 0$ and expanding if $\lambda <0.$  The function $f$ is called the potential function.

Now, by setting $\overline f(t)=\psi_t^*f$, using \eqref{self-similar:sol1:alt}, \eqref{self-similar:sol2:alt} and conformal theory, we obtain
 \begin{align*}
\left\{
\begin{array}{lcl}
\Ric_{\overline g} + \nabla^2_{\overline g}\overline f - \alpha_n d\overline w  \otimes d\overline w = \frac{\lambda}{\sigma(t)}\overline g,\\
\Delta_{\overline g}\overline w = \langle\nabla_{\overline g}\overline f, \nabla_{\overline g}\overline w\rangle_{\overline g}. 
\end{array}
\right.
\end{align*}
Moreover, by scaling $\overline g$ one can normalize $\lambda = 1/2$ in the shrinking case and, $\lambda=-1/2$ in the expanding case. For $\lambda = 1/2,$ $\sigma (t) = 1 - t > 0$ implies $t < 1.$ Setting $s = t - 1,$ we have  $s + 1 = t < 1,$ i.e., $s < 0$, and then, $\overline{g}(s)=\sigma(s)\psi^*_sg$, with $\sigma (s) = -s$ and $\psi_{-1} = \Id.$ For $\lambda = -1/2,$ $\sigma (t) = 1 + t > 0$ implies $t >- 1.$ Setting $s = t + 1,$ we have $s - 1 = t >-1,$ i.e., $s > 0,$ and then, $\overline{g}(s)=\sigma(s)\psi^*_sg$, with $\sigma (s) = s$ and $\psi_{1} = \Id.$ Hence, we immediately obtain the next proposition.
\begin{proposition}\label{self:similar:solution}
Consider an $n(\geqslant 3)$-dimensional smooth manifold $M$, and let $(\overline g(t), \overline w(t))$ be a gradient soliton to the extended Ricci flow on $M$. The following identities hold for all time $t$:
\begin{subequations}
\begin{empheq}[left=\empheqlbrace]{align}
&\Ric_{\overline g} + \nabla^2_{\overline g}\overline f - \alpha_n d \overline w  \otimes d \overline w = \frac{c}{2t}\overline g,\label{self-similar:sol1}\\
&\Delta_{\overline g} \overline w = \langle\nabla_{\overline g} \overline f, \nabla_{\overline g} \overline w\rangle_{\overline g}\label{self-similar:sol2},
\end{empheq}
\end{subequations}
 where $c  =  0$ in the steady case (for $t \in  \mathbb{R}$ and $\psi_0=\Id$), $c = -1$ in the shrinking case (for $t \in (-\infty, 0)$ and $\psi_{-1}=\Id$) and $c = -1$ in the expanding case (for $t \in (0,+\infty)$ and $\psi_1=\Id$), besides
 \begin{align}\label{self-solution1}
\frac{\partial}{\partial t} \overline f = \|\nabla_{\overline g} \overline f\|^2_{\overline g}.
 \end{align}  
The function $\overline f$ is still called the potential function.
\end{proposition}

Now, we will show how to construct a family of mean curvature solitons and establish a characterization of such a family. For it, let $M$ be an $n(\geqslant 3)$-dimensional smooth manifold, and let $(\overline{g}(t), \overline{w}(t))$ be a gradient soliton to the extended Ricci flow on $M$ for some initial value $(g,w)$ and with potential function $\overline f=\psi^*_tf$, where $\{\psi_t\}$ is the smooth one-parameter family of diffeomorphisms of $M$ generated by $Y_t = \frac{\nabla_g f}{\sigma(t)}$, with $\sigma(t) = -\kappa t$ and $\psi_{-\kappa} = \Id,$ where $\kappa = 1$ in the shrinking case (for $t \in (-\infty, 0$)), $\kappa = -1$ in the expanding case (for $t \in (0,+\infty)$) and $\sigma(t) = 1$ in the steady case (for $t \in  \mathbb{R})$ with  $\psi_{0} = \Id$ (see Proposition~\ref{self:similar:solution}). 

Given an $(n-1)$-dimensional smooth compact manifold $\Sigma$ without boundary, let $\{x(\,\cdot\,,t)\}$ be a smooth one-parameter family of immersions of $\Sigma$ into $M$, where $x(\,\cdot\,,t):=\psi(\,\cdot\,,-t-2\kappa)$ and $x(\,\cdot\,,t):=\psi(\,\cdot\,,-t)$ in the steady case. Note that $x(\,\cdot\,,-\kappa)=\psi(\,\cdot\,,-\kappa)=\Id$ and $x(\,\cdot\,,0)=\psi(\,\cdot\,,0)=\Id$. Moreover, when considering $x(\cdot, t) :=\psi(\cdot,-t-2\kappa),$ we are assuming $t \in (-2,0)$ in the shrinking case,  $t \in (0, 2)$ in the expanding case, and $t\in \mathbb{R}$ in the steady case. For each $t,$ set  $x_t = x(\,\cdot\,, t),$ $\Sigma_t$ for the hypersurface $x_t(\Sigma)$ of $(M,\overline g(t)),$ i.e., $\Sigma_t:=(x_t(\Sigma), \overline g(t)),$ and $\mathscr G :=\{\Sigma_t\}.$ In particular, if $\mathscr G$ evolves by MCF in the $(\overline g, \overline w)$-extended Ricci flow background on $M$, then it is a family of mean curvature solitons. 
Indeed, since $\overline g(t)=\sigma(t)\psi_t^*g$, we have $\nabla_g  f= \sigma(t)\nabla_{\overline g(t)}  \overline f$, and then
\begin{align*}
 H(p,t) &= \overline g(t)\Big( \frac{\partial}{\partial t} x(p,t),  e(p, t)\Big) =-\overline g(t)\Big(  \frac{ \nabla_g  f(p)}{\sigma(t)},  e(p, t)\Big) \\
 &= -\overline g(t)\Big(   \nabla_{\overline g(t)}  \overline f(p),  e(p, t)\Big) =- e(p, t) \overline f(p),
\end{align*}
it proves our claim. A sufficient condition for ensuring that $\mathscr G$ is a family of mean curvature solitons is that the hypersurface $\Sigma$ must be $f$-minimal. Besides, we will see that any family $\mathscr F$ of mean curvature solitons is given by the family $\mathscr G$ up to reparametrization, as stated below.
\begin{theorem}\label{NazasCharacterization}
If $\Sigma$ is a $f$-minimal hypersurface of $(M,g)$, then $\mathscr G$ is a family of mean curvature solitons. Moreover, any family $\mathscr F$ of mean curvature solitons is given by $\mathscr G$ up to reparametrization.
\end{theorem}

\begin{proof}
Let $\Sigma$ be a hypersurface of $(M,g)$ satisfying $H + e_o f = 0$ on $\Sigma,$ where $e_{o}$ is the  unit normal vector field on $\Sigma$. Take $\mathscr G =\{\Sigma_t\}$ the smooth one-parameter family of isometric immersions of $\Sigma$ into $M$ as above, so that $e_o = \sqrt{\sigma (t)} e(\,\cdot \,, t),$ and then $ \mathcal A_{e_o} = \sqrt{\sigma (t)}  \mathcal A_{e(\cdot,t)} $ that implies $H = \sqrt{\sigma(t)} H(\cdot, t).$ So, $H(\,\cdot\,,t) + e(\,\cdot\,, t) \overline f=0$. Thus,
\begin{align*}
 \Big(\frac{\partial}{\partial t}  x (\,\cdot\,, t)\Big)^\perp &= \overline g(t)\Big(\frac{\partial}{\partial t}  x (\,\cdot\,, t), e(\,\cdot\,,t)\Big)e(\,\cdot\,,t)
 = -\overline g(t)\Big(  \frac{\nabla_g f}{\sigma(t)}, e(\,\cdot\,, t) \Big)e(\,\cdot\,,t) \\
 &= -\overline g(t)\Big(   \nabla_{\overline g(t)}  \overline f,  e(\,\cdot\,, t)\Big) e(\,\cdot\,,t)=- e(\,\cdot\,, t) (\overline f) e(\,\cdot\,, t)  = H(\,\cdot\,,t) e(\,\cdot\,, t).   
\end{align*}
Now, we affirm that if a smooth family of hypersurfaces $\Sigma_t = x_t(\Sigma)$ satisfies $\langle\frac{\partial}{\partial t} x(p,t), e(p,t)\rangle =  H(p,t)$, then it can be everywhere locally reparametrized to a mean curvature flow. Indeed, if $\frac{\partial}{\partial t} x(p,t) = H(p,t) e(p,t) + X(p,t), $ where $X(p,t) \in d x_t(Tp\Sigma)\,\,\,\forall p \in \Sigma,$ take $\{\varphi_t\}$ the smooth one-parameter family of diffeomorphisms of $\Sigma$ generated by $Y(p,t) =  -[dx_t]^{-1}(X(p,t))$ and then consider the reparametrization $\widetilde x(p,t) = x(\varphi_t (p), t).$ By a straightforward computation  
 $\{\widetilde \Sigma _t := \widetilde x_t(\Sigma)\}$ evolves by MCF  in the $(\overline g, \overline w)$-extended Ricci flow background on $M$. Finally, by a simple analysis of this proof, we also show that any family $\mathscr F$ of mean curvature solitons is given by $\mathscr G$ up to reparametrization.
\end{proof}

We finalize this section by proving our Theorem~\ref{Huisken_monotonicity}. We begin determining how evolves the area of a mean curvature flow in an extended Ricci flow background.

\begin{lemma}\label{Huisken:lem}
 Let $(\overline{g}(t),\overline{w}(t)),$ $\overline f$ and $\mathscr F:=\{\Sigma_t\}$ be as in the statement of Theorem~\ref{Huisken_monotonicity}. Then, the following equation holds on $\Sigma_t:$ 
\begin{align*}
\dfrac{d}{dt}(d A_{\overline{g}}) = - (\overline{R}^i\!_i + H_{\overline{g}}^2 - \alpha_n |\widehat{\nabla}_{\overline{g}} \overline{w} |_{\overline{g}}^2) d A_{\overline{g}}.
\end{align*}
\end{lemma}
\begin{proof}
The lemma follows by using the well-known formula
\begin{align*}
\dfrac{d}{dt}(d A_{\overline{g}}) = \dfrac{1}{2}\mbox{tr}_{(\overline{g}_{ij}(t))} \Big( \dfrac{\partial}{\partial t} \overline{g}_{ij}\Big) dA_{\overline{g}}
\end{align*}
and equation~\eqref{List:together:mean:curv} in Proposition~\ref{mean_curv_flow_in_a_mod} (see also Remark~\ref{without}).
\end{proof}

\begin{proof}[\bf Proof of Theorem~\ref{Huisken_monotonicity}]
Lemma~\ref{Huisken:lem} and a straightforward computation  yield
\begin{align*}
    \dfrac{d}{dt}\int_{\Sigma_t} e^{-\overline{f}} d A_{\overline{g}} &= - \int_{\Sigma_t} \Big(\frac{d}{dt}\overline{f} + \overline{R}^i\!_i + H_{\overline{g}}^2 - \alpha_n |\widehat{\nabla}_{\overline{g}} \overline{w} |_{\overline{g}}^2\Big) e^{-\overline{f} } d A_{\overline{g}}.
\end{align*}
Therefore,
\begin{align*}
\dfrac{d}{dt}\int_{\Sigma_t} e^{-\overline{f}} d A_{\overline{g}}  &= - \int_{\Sigma_t} \Big(\dfrac{\partial}{\partial t}\overline{f} + H_{\overline{g}} e_t \overline{f}  + \overline{R}^i\!_i + H_{\overline{g}}^2 - \alpha_n |\widehat{\nabla}_{\overline{g}} \overline{w} |_{\overline{g}}^2\Big) e^{-\overline{f}} d A_{\overline{g}}.
\end{align*}

First, assume $(\overline{g}(t), \overline{w}(t))$ is a gradient steady soliton. In this case, we can take traces in the first equation 
of~\eqref{mod_grad_Ricci_soliton}  on $\Sigma_t$ to get
\begin{align*}
0=\overline{R}^i\!_i + \overline{\nabla}_i \overline{\nabla}^i \overline{f} -\alpha_n  |\widehat{\nabla}_{\overline{g}} \overline{w}|_{\overline{g}}^2=\overline{R}^i\!_i + \widehat{\nabla}^i \widehat{\nabla}_i \overline{f} - H_{\overline{g}} e_t \overline{f} -\alpha_n  |\widehat{\nabla}_{\overline{g}} \overline{w}|_{\overline{g}}^2.
\end{align*}
Then, using \eqref{self-solution}, we obtain
\begin{align*}
\dfrac{d}{dt}\int_{\Sigma_t} e^{-\overline{f}} d A_{\overline{g}}
&= - \int_{\Sigma_t} \Big(|\nabla_{\overline{g}} \overline{f}|_{\overline{g}}^2 - \widehat{\Delta}_{\overline{g}}\overline{f} +2 H_{\overline{g}} e_t \overline{f}  + H_{\overline{g}}^2 \Big)e^{-\overline{f}} d A_{\overline{g}} \\
&= - \int_{\Sigma_t} \Big( |\widehat{\nabla}_{\overline{g}} \overline{f}|_{\overline{g}}^2 +(e_t \overline{f})^2 - \widehat{\Delta}_{\overline{g}}\overline{f} + 2H_{\overline{g}} e_t \overline{f}  + H_{\overline{g}}^2\Big)e^{-\overline{f}} d A_{\overline{g}} \\
&= - \int_{\Sigma_t} \Big(H_{\overline{g}} +e_t \overline{f} \Big)^2 e^{-\overline{f}} d A_{\overline{g}},
\end{align*}
where in the second line we have used  the equality
$$\widehat{\Delta}_{\overline{g}} e^{- \overline{f}} = 
(|\widehat{\nabla}_{\overline{g}} \overline{f}|_{\overline{g}}^2 - 
\widehat{\Delta}_{\overline{g}}\overline{f})e^{-\overline{f}}$$ 
and Stokes' theorem. Since the boundary integrand in the right-hand side 
is nonnegative, we have immediately the result of the theorem for the steady case.

For the shrinking case, we claim that the function
$$(-\infty,0)\ni t\mapsto \tau^{-( n - 1 ) / 2}\int_{\Sigma_t} e^{-\overline{f}} dA_{\overline{g}}$$ 
is non increasing, where $\tau = -t$. Indeed, 
as above, we take traces in the first equation 
of~\eqref{mod_grad_Ricci_soliton} on $\Sigma_t$ to obtain
\begin{align*}
- \dfrac{n - 1}{2t}=\overline{R}^i\!_i + 
\overline{\nabla}^i \overline{\nabla}_i \overline{f} -
\alpha_n  |\widehat{\nabla}_{\overline{g}} \overline{w}|_{\overline{g}}^2 
= \overline{R}_i^i + \widehat{\nabla}^i \widehat{\nabla}_i \overline{f}- 
H_{\overline{g}} e_t \overline{f} -\alpha_n  |\widehat{\nabla}_{\overline{g}} \overline{w}|_{\overline{g}}^2.
\end{align*}
Then,
\begin{align*}
&\dfrac{d}{dt}\Big(\tau^{-( n - 1 ) / 2}\int_{\Sigma_t} e^{-\overline{f}} d A_{\overline{g}}\Big)\\
&= - \tau^{-( n - 1 ) / 2}\int_{\Sigma_t} \Big( |\widehat{\nabla}_{\overline{g}} \overline{f}|_{\overline{g}}^2 +(e_t \overline{f})^2 - \widehat{\Delta}_{\overline{g}}\overline{f} + 2H_{\overline{g}} e_t \overline{f}  + H_{\overline{g}}^2- \dfrac{n - 1}{2t}\Big)e^{-\overline{f}} d A_{\overline{g}} \\
&\quad+\frac{ n - 1 }{2}\tau^{-\frac{( n - 1 )}{2} - 1}\int_{\Sigma_t} e^{-\overline{f}} dA_{\overline{g}}\\
&= -\tau^{-( n - 1 ) / 2} \int_{\Sigma_t} \Big(H_{\overline{g}} +e_t \overline{f} \Big)^2 e^{-\overline{f}} d A_{\overline{g}}.
\end{align*}
This proves the claim, and so the theorem for the shrinking case. Finally, in a similar way, one proves the expanding case.
\end{proof}

\begin{remark} \label{rem: recover5}
For the shrinking case in Theorem~\ref{Huisken_monotonicity},  we recover 
Huisken's monotonicity formula~\cite[Thm.~3.1]{Huisken1990}, 
by taking $M = \mathbb{R}^n$, $g_{\alpha\beta}(\tau) = \delta_{\alpha \beta}$, $\overline{f}(x, \tau) = |x|^2/(4 \tau)$ 
and $\overline{w}(\tau)= w$ constant.
\end{remark}

\section{Examples of soliton solutions to the extended Ricci flow}

In this section, we show how to obtain a gradient soliton solution to the extended Ricci flow, and then we are obtaining its corresponding extended Ricci flow and explicit examples of mean curvature flow in an extended Ricci flow background (see Propositions~\ref{self:similar:solution:alt} and~\ref{self:similar:solution} and Theorem~\ref{NazasCharacterization}). For explicit examples of mean curvature flow in a Ricci flow background, see the work by Yamamoto~\cite{yamamoto2018examples}.

Let $g = \frac{1}{F^2}g_0$ be a Riemannian metric on $\mathbb{R}^n$, where $g_0$ stands for the Euclidean metric and $F$ is a nonzero smooth function on $\mathbb{R}^n$, and  consider 
\begin{subequations}\label{Bizu}
\begin{empheq}[left=\empheqlbrace]{align}
&Ric_{g} + \nabla^2_{g} f - \alpha_n d w  \otimes d w = \lambda g,\label{Bizu1}\\
&\Delta_{g} w = \langle\nabla_{g} f, \nabla_{g} w\rangle_{g}\label{Bizu2}.
\end{empheq}
\end{subequations}
Since the metric $g$ is conformal to $g_0$, we have 
\[Ric_{g} =\frac{1}{F^2}\Big((n - 2)F\nabla^2  F + (F\Delta F - (n - 1)|\nabla F|^2)g_0\Big)\]
and the following equations are valid
\begin{align*}
(\nabla^2_g h)_{ij} &= h_{x_ix_j} +\frac{F_{x_j}}{F} h_{x_i} +\frac{F_{x_i}}{F}
h_{x_j} \quad \rm{for} \quad i \neq j,\\
(\nabla^2_g h)_{ii} &= h_{x_ix_i} + 2\frac{F_{x_i}}{F} h_{x_i} -\sum_k\frac{F{x_k}}{F} h_{x_k} \quad \rm{for}\quad i = j,  
\end{align*}
for any smooth function $h$ on $\mathbb{R}^n$. Hence,
$$\Delta_{g} h = F^2 \big(\sum_k h_{x_kx_k} + (2 - n) \frac{1}{F} \sum_kF_{x_k}h_{x_k}\big).$$

We find solutions of Eq.~\eqref{Bizu1} (and thus of~\eqref{Bizu2}) of the form $f (\xi )$ and $w(\xi )$, that is, they only depend on $\xi =\sum_{i=1}^n \alpha_ix_i$ with $\alpha_i \in \mathbb{R}$ and $\sum_{i = 1}^n \alpha_i^2 = 1$. The following proposition provides the system of ordinary differential equations that must be satisfied by such solutions so that we can obtain all parameters necessary to construct gradient soliton solution to extended Ricci flow on $M$.
\begin{proposition}\label{Feitosa}
Let $\mathbb{R}^n$, with $n \geqslant 3$, be an Euclidean space with coordinates $x = (x_1, \ldots , x_n)$ and metric $g = \frac{1}{F^2(\xi)}g_0$, where $F (\xi ) \in C^\infty(\mathbb{R}^n)$, $\xi  = \sum_{i = 1}^n \alpha_ix_i$ with $\alpha_i \in \mathbb{R}$ and $\sum_{i = 1}^n \alpha_i^2 = 1$. We can obtain smooth functions  $f (\xi )$ and $ w(\xi )$ satisfying~\eqref{Bizu1} (and thus~\eqref{Bizu2} as well) by means of the equation
\begin{align*}
\frac{1}{F}  F''- \frac{n - 1}{F^2} {F'}^2 +\frac{1}{n - 2}{f'}^2 - \frac{w''f'}{(n - 2)w'} = \frac{\lambda}{F^2}.
\end{align*}
\end{proposition}
\begin{proof}
We need to analyze Eq.~\eqref{Bizu1} in two cases. For $i \neq j$, it rewrites as
\begin{align}\label{Holds}
(n - 2) \frac{F_{x_ix_j}}{F} + f_{x_i x_j} +\frac{F_{x_j}}{F} f_{x_i} +\frac{F_{x_i}}{F}f_{x_j} - \alpha_n w_{x_i} w_{x_j} = 0
\end{align}
and for $i=j$,
\begin{align}\label{BIZU1}
(n - 2) \frac{F_{x_ix_i}}{F} + \!\sum_k\!\Big( \frac{ F_{x_k x_k} }{F}
-\! \frac{n - 1}{F^2}F^2_{x_k} -\!\frac{F_{x_k}}{F}f_{x_k}\Big) + f_{x_i x_i} +2\frac{F_{x_i}}{F} f_{x_i}  -\! \alpha_n w^2_{x_i}  \!=\! \frac{\lambda}{F^2}.
\end{align}
On the other hand, equation~\eqref{Bizu2} becomes
\begin{align}\label{add:Naza}
F^2 \big(\sum_k w_{x_kx_k} + (2 - n) \frac{1}{F} \sum_kF_{x_k}w_{x_k}\big) = F^2 \sum_kf_{x_k}w_{x_k}.
\end{align}
We now assume that the argument $\xi$ of the functions $F(\xi )$, $f (\xi )$ and $w(\xi)$ is of the form $\xi =\sum^n_{i = 1} \alpha_i x_i$. Hence, we have $F_{x_i}= F'\alpha_i$ and $F_{x_ix_j} = F''\alpha_i\alpha_j$ where the superscript $'$ denotes the derivative with respect to $\xi$.
Using the same reasoning for $f$ and $w$, equations~\eqref{Holds} and~\eqref{BIZU1} become
\begin{align}\label{Sys:Bizu1}
(n - 2) \frac{F''}{F} + f'' +2\frac{F'}{F} f'- \alpha_n {w'}^2 = 0
\end{align}
and
\begin{align}\label{Sys:Bizu2}
&(n - 2) \frac{F''}{F}\alpha^2_i + \sum_k\Big( \frac{F''}{F}   \alpha_k^2 - \frac{n - 1}{F^2}{F'}^2\alpha_k-\sum_k\frac{F'}{F}f'\alpha^2_k\Big) + f''\alpha_i^2 \\
&\quad +2\frac{F'}{F} f'\alpha^2_i - \alpha_n {w'}^2 \alpha_i^2 = \frac{\lambda}{F^2}.\nonumber
\end{align}
Since $n \geqslant 3$, we can choose this invariance so that at least two indices $i, j$ are such that $\alpha_i\alpha_j \neq 0$ and $\sum^n_{i=1} \alpha^2_i = 1$, and then equations~\eqref{Sys:Bizu1} and \eqref{Sys:Bizu2} become
\begin{align}\label{Guerraesua}
(n - 2) \frac{F''}{F} + f'' +2\frac{F'}{F} f'- \alpha_n {w'}^2 = 0
\end{align}
and
\begin{align}\label{Guerraesua1}
(n - 2) \frac{F''}{F}\alpha^2_i + \frac{1}{F}  F''- \frac{n - 1}{F^2} {F'}^2-\frac{F'}{F}f'+ f''\alpha_i^2 +2\frac{F'}{F} f'\alpha^2_i  - \alpha_n {w'}^2 \alpha_i^2 = \frac{\lambda}{F^2 }.
\end{align}
 Plugging~\eqref{Guerraesua} into~\eqref{Guerraesua1}, one has
\begin{align}\label{Vamosla}
\frac{1}{F}  F''   
- \frac{n - 1}{F^2} {F'}^2 -\frac{F'}{F}f' = \frac{\lambda}{F^2}.
\end{align}
Eq.~\eqref{add:Naza} provides $w'' -(n - 2)\frac{1}{F} F'w' =  f'w'.$ Assuming $w' \neq 0$ and using~\eqref{Vamosla}, we obtain
\begin{align*}
\frac{1}{F}  F''   
- \frac{n - 1}{F^2} {F'}^2 +\Big(\frac{1}{n - 2}f' - \frac{w''}{(n - 2)w'}\Big)f' = \frac{\lambda}{F^2}.
\end{align*}
Therefore
\begin{align*}
\frac{1}{F}  F''   
- \frac{n - 1}{F^2} {F'}^2 +\frac{1}{n - 2}{f'}^2 - \frac{w''f'}{(n - 2)w'} = \frac{\lambda}{F^2}.
\end{align*}
This finishes the proof of the proposition.
\end{proof}
\begin{remark}
For constructing a family of mean curvature solitons $\mathscr G$ in the $(\overline g, \overline w)$-extended Ricci flow background on $(\mathbb R^n,\frac{1}{F^2(\xi)}g_0)$, it is enough to consider a $f$-minimal hypersurface $\Sigma$ in this geometric ambient space. Indeed, it follows immediately from Propositions~\ref{self:similar:solution:alt} and~\ref{self:similar:solution}, Theorem~\ref{NazasCharacterization} and Proposition~\ref{Feitosa}.
\end{remark}

In what follows, we are using Proposition~\ref{Feitosa} to show how to obtain explicit parameter functions for constructing gradient soliton solutions to the extended Ricci flow on $(\mathbb{R}^n, \frac{1}{F^2(\xi)}g_0).$ 

\begin{example}
For $f(\xi) = e^\xi$ and $F(\xi) = e^{-\xi},$ we have
\begin{align*}
1 - (n - 1) +\frac{e^{2\xi}}{n - 2} - \frac{w''e^\xi}{(n - 2)w'} =  \lambda  e^{2\xi}.
\end{align*}
Therefore
\begin{align*}
\frac{w''}{w'} = - (n - 2)^2  e^{-\xi} - \lambda (n - 2) e^{\xi} +   e^{\xi}, 
\end{align*}
and   then
\begin{align*}
\ln w' = (n - 2)^2  e^{-\xi}  -\lambda(n - 2) e^{\xi} + e^{\xi} + c,
\end{align*}
for some constant $c$. Hence,
\begin{align*}
w = \int e^{[(n - 2)^2  e^{-\xi}  -\lambda(n - 2) e^{\xi} + e^{\xi} + c]} d \xi. 
\end{align*}
\end{example}

\begin{example}
For $f(\xi) = \tan{\xi} $ and $F(\xi) =  \cot{\xi},$ with $ 0 <\xi < \frac{\pi}{2},$ we have  
\begin{align*}
2\csc{\xi} -  \frac{4(n - 1)}{\sin^2{2\xi}} +\frac{\sec^4{\xi}}{n - 2} - \frac{w''\sec^2{\xi}}{(n - 2)w'} = \lambda \tan^2{\xi}.
\end{align*}
Therefore,
\begin{align*}
\frac{w''\sec^2{\xi}}{(n - 2)w'} = 2\csc{\xi} -  \frac{4(n - 1)}{\sin^2{2\xi}} +\frac{\sec^4{\xi}}{n - 2} -  \lambda \tan^2{\xi}.
\end{align*}
Hence,
\begin{align*}
\frac{w''}{w'} &=\frac{n - 2}{\sec^2{\xi}}\Big( 2\csc{\xi} -  \frac{4(n - 1)}{\sin^2{2\xi}} +\frac{\sec^4{\xi}}{n - 2} -  \lambda \tan^2{\xi}\Big)\\
&= 2(n - 2)\cot{\xi} -  (n - 1) (n - 2) \csc{\xi} +\sec^2{\xi} -  \lambda (n - 2) \sin^2{\xi}.
\end{align*}
Whence, 
\begin{align*}
\ln w' = \int \Big( 2(n - 2)\cot{\xi} -  (n - 1) (n - 2) \csc{\xi} +\sec^2{\xi} -  \lambda (n - 2) \sin^2{\xi}\Big) d \xi + c,
\end{align*}
for some constant $c$. Thus,
\begin{align*}
w' &= e^{\int \big( 2(n - 2)\cot{\xi} -  (n - 1) (n - 2) \csc{\xi} +\sec^2{\xi} -  \lambda (n - 2)\sin^2{\xi}\big) d \xi + c}\\
 &= e^{ 2(n - 2)\ln{\sin{\xi}} -  (n - 1) (n - 2) \big(\ln{\sin{\frac{\xi}{2}}} -\ln{\cos{\frac{\xi}{2}}}\big) + \tan{\xi} -  \lambda (n - 2) \big( \frac{1}{2}\xi -\frac{1}{2}\sin \xi \cos \xi \big) + c}.
 \end{align*}
 So,
 \begin{align*}
 w &= \int e^{ 2(n - 2)\ln{\sin{\xi}} -  (n - 1) (n - 2) \big(\ln{\sin{\frac{\xi}{2}}} -\ln{\cos{\frac{\xi}{2}}}\big) + \tan{\xi} -  \lambda (n - 2) \big( \frac{1}{2}\xi -\frac{1}{2}\sin \xi \cos \xi \big) + c} d \xi. 
 \end{align*}
\end{example}

\begin{example}\label{sphere}
Let $\mathbb{B}_+^n\subset \mathbb{R}^n, \,n\geqslant3,$ be a unitary upper half ball with metric  $g = \frac{1}{(1 +  x_{n})^2}g_0.$ Note that its boundary is the standard unitary sphere $(\mathbb{S}^{n-1},g_0)$, $\xi = x_n$ and $F(x_n) = 1 +  x_{n}$. Moreover, the mean curvature of $\,(\mathbb{S}^{n - 1}, g_0)$ with respect to $e_0 = -e_n$  is $H_{g_0}= n - 1,$ so that we  can take $f(x) = (n - 1)\langle x, e_n\rangle =(n - 1) x_n,$ since $H_{g_0} + e_0 f = 0.$ By Proposition~\ref{Feitosa}, 
\begin{align*}
\frac{1}{F}  F''- \frac{n - 1}{F^2} {F'}^2 +\frac{1}{n - 2}{f'}^2 - \frac{w''f'}{(n - 2)w'} = \frac{\lambda}{F^2}.
\end{align*}
Since $F' = 1$ and $F'' = 0$, we get
\begin{align*}
 \frac{(n - 1)w''}{(n - 2)w'} = - \frac{n - 1}{(1 +  x_{n})^2}  +\frac{(n - 1)^2}{n - 2}  - \frac{\lambda}{(1 +  x_{n})^2}.
\end{align*}
So,
\begin{align*}
 \frac{w''}{w'} = - \frac{n - 2}{(1 +  x_{n})^2}  + n - 1  - \frac{\lambda (n - 2)}{(n - 1)(1 +  x_{n})^2 }.
\end{align*}
Whence,
\begin{align*}
\ln w' &= \frac{n - 2}{1 +  x_{n}}  + (n - 1)x_n  + \frac{\lambda (n - 2)}{(n - 1)(1 +  x_{n}) } + c,
\end{align*}
for some constant $c$, and then
\begin{align*}
 w &= \int e^{\frac{n - 2}{1 +  x_{n}}  + (n - 1)x_n  + \frac{\lambda (n - 2)}{(n - 1)(1 +  x_{n}) } + c} dx_n.
\end{align*}
\end{example}

\section{Acknowledgements}
The authors would like to express their sincere thanks to Lucas Ambrozio (IMPA), Abdênago Barros (UFC), Valter Borges (UFPA), Ronaldo de Lima (UFRN) and Marcus Marrocos (UFAM) for useful comments, discussions and constant encouragement. The second author is also grateful to Department of Mathematics at Universidade Federal de São Carlos for a good atmosphere while this work was done. José N. V. Gomes has been partially supported by Conselho Nacional de Desenvolvimento Científico e Tecnológico (CNPq), of the Ministry of Science, Technology and Innovation of Brazil, Grant 310458/2021-8. Matheus Hudson has been partially supported by Fundação de Amparo à Pesquisa do Estado do Amazonas (FAPEAM), Grant 062.00931/2013.

\end{document}